%% file: ms.tex
\documentclass[review,onefignum,onetabnum]{siamart190516}

\input{ex_shared}

\ifpdf
\hypersetup{
  pdftitle={Homogenization of a Non-dilute Suspension of a Viscous Fluid with Magnetic Particles},
  pdfauthor={T. Dang, Y. Gorb and S. Jim\'{e}nez Bola\~{n}os}
}
\fi

\nolinenumbers
\begin{document}

\maketitle

\begin{abstract}
This paper seeks to carry out the rigorous homogenization of a particulate flow consisting of a non-dilute suspension of a viscous Newtonian fluid with magnetizable particles. The fluid is assumed to be described by the Stokes flow, while the particles are either \textit{paramagnetic} or \textit{diamagnetic}, for which the magnetization field is a linear function of the magnetic field. The coefficients of the corresponding partial differential equations are locally periodic. A one-way coupling between the fluid domain and the particles is also assumed. The homogenized or effective response of such a suspension is derived, and the mathematical justification of the obtained asymptotics is carried out. The two-scale convergence method is adopted for the latter. As a consequence, the presented result provides a justification for the formal asymptotic analysis of L\'{e}vy and Sanchez-Palencia \cite{levySuspensionSolidParticles1983} for particulate steady-state Stokes flows.
\end{abstract}



\input{content}

\section*{Acknowledgments}
The work of the first two authors was partially supported by NSF grant DMS-1350248. This material is based upon work supported by and while serving at the National Science Foundation for the second author Y. Gorb. Any opinion, findings, and conclusions or recommendations expressed in this material are those of the authors and do not necessarily reflect views of the National Science Foundation.

\bibliographystyle{siamplain}
\bibliography{references.bib}

\end{document}

%% file: ex_shared.tex
\usepackage{lipsum}
\usepackage{amsfonts}
\usepackage{graphicx}
\usepackage{epstopdf}
\usepackage{algorithmic}
\ifpdf
  \DeclareGraphicsExtensions{.eps,.pdf,.png,.jpg}
\else
  \DeclareGraphicsExtensions{.eps}
\fi

\usepackage{bm}
\usepackage{mathtools} 
\usepackage{enumitem} 
\usepackage{comment}
\usepackage{marginnote}
\reversemarginpar
\usepackage{stmaryrd} 
\usepackage{dsfont} 
\usepackage{import} 
\allowdisplaybreaks

\newcommand{\RR}{\mathbb{R}}

\newcommand{\ZZ}{\mathbb{Z}}

\newcommand{\innerp}[1]{\left\langle #1 \right\rangle}
\newcommand{\norm}[1]{\left\Vert #1 \right\Vert}
\newcommand{\abs}[1]{\left\vert #1 \right\vert}

\newcommand{\DD}{\bm{D}} 
\newcommand{\ee}{\bm{e}}

\newcommand{\bg}{\bm{g}}

\newcommand{\kk}{\bm{k}}
\newcommand{\calH}{\mathcal{H}}
\newcommand{\calK}{\mathcal{K}}
\newcommand{\calL}{\mathcal{L}}
\newcommand{\II}{\bm{I}}

\newcommand{\nn}{\bm{n}}
\newcommand{\bP}{\bm{U}} 
\newcommand{\bQ}{\bm{Q}}

\newcommand{\calP}{\mathcal{P}}
\newcommand{\bT}{\bm{T}}
\newcommand{\us}{\mathrm{s}}

\newcommand{\uu}{\bm{u}}
\newcommand{\calA}{\mathcal{N}}
\newcommand{\calU}{\mathcal{U}}

\newcommand{\calV}{\mathcal{V}}

\newcommand{\calQ}{\mathcal{Q}}

\newcommand{\vv}{\bm{v}}
\newcommand{\ww}{\bm{w}}

\newcommand{\xx}{\bm{x}}

\newcommand{\yy}{\bm{y}}

\newcommand{\zz}{\bm{z}}

\newcommand{\emB}{\mathbf{B}}

\newcommand{\emH}{\mathbf{H}}

\newcommand{\tscale}{\xrightharpoonup[]{~2~}}
\newcommand{\per}{\mathrm{per}}

\newcommand{\eff}{\mathrm{eff}}
\newcommand{\nRe}{{R}_{{e}}}

\renewcommand\vec[1]{\bm #1}
\DeclareMathOperator{\supp}{supp}

\DeclareMathOperator*{\argmin}{arg\,min}

\newcommand{\charac}{\mathds{1}}

\providecommand\given{} 
\newcommand\givensymbol[1]{%
  \nonscript\;\delimsize#1\allowbreak\nonscript\;\mathopen{}%
}
\DeclarePairedDelimiterX\Set[1]\{\}{%
  \renewcommand\given{\givensymbol{\vert}}%
  #1%
}

\newsiamremark{remark}{Remark}
\newsiamremark{hypothesis}{Hypothesis}
\crefname{hypothesis}{Hypothesis}{Hypotheses}
\newsiamthm{claim}{Claim}

\headers{Homogenization of Non-dilute Magnetic Suspensions}{T. Dang, Y. Gorb, and S. Jim\'{e}nez Bola\~{n}os}
\title{Homogenization of Non-dilute Suspension of Viscous Fluid with Magnetic Particles\thanks{Submitted to the editors DATE.
\funding{The work of the first two authors was partially funded by DMS-1350248.}}}

\author{
   Thuyen Dang\thanks{Department of Mathematics, University of
    Houston, 4800 Calhoun Road, Houston, Texas 77204
    (\email{ttdang9@central.uh.edu}).} 
  \and
  Yuliya Gorb\thanks{National Science Foundation, 2415 Eisenhower Avenue, Alexandria, Virginia 22314
    (\email{ygorb@nsf.gov}).}
    \and
  Silvia Jim\'{e}nez Bola\~{n}os\thanks{Department of
    Mathematics, Colgate University, 13 Oak Drive, Hamilton, New York 13346
    (\email{sjimenez@colgate.edu}).} 
}

\usepackage{amsopn}

\DeclareMathOperator{\di}{d\!}
\DeclareMathOperator{\Div}{\mathrm{div}}



%% file: content.tex
\section{Introduction}
\label{sec:introduction-1}

The past two decades have witnessed a dramatic growth of research interest in the \textit{ferrofluids} and \textit{magnetorheological fluids}, e.g. \cite{afkhamiFerrofluidsMagneticallyGuided2017,alougesHomogenizationCompositeFerromagnetic2015,eloulReactiveMomentumTransfer2020,kangDirectSimulationDynamics2013,kangDirectSimulationMethod2008}, that found abundant technological, environmental and medical applications. The most salient applications are magnetic drug targeting \cite{itoMedicalApplicationFunctionalized2005,voltairasHydrodynamicsMagneticDrug2002}, and molecular communication using magnetic nanoparticles \cite{nakanoMolecularCommunicationNetworking2012} for the former, and non-invasive measurements of blood pressure, prosthetic knee, and many others (see
\cite{coonApplicationMagnetoRheologicalFluids2019,gadekarMagnetorheologicalFluidIts} as well as references therein) for the latter. 
Both of these types of \textit{fluids} are colloids consisting of a large number of fairly small magnetizable particles dispersed/suspended in a carrier fluid, which is itself electrically nonconducting, in the presence of strong magnetic fields or strong magnetic gradients (a typical ferro- or magnetorheological fluid also includes a surfactant layer that for simplicity will be ignored hereafter). Under those external forces, the particles slip relative to the ambient fluid and, thereby, exert drag to the latter yielding the overall suspension to move as a whole.  The most important applications of the aforementioned \textit{fluids} result from the possibility of controlling their effective viscosity with an externally applied magnetic field (e.g., magnetorheological fluids can even solidify in the presence of a magnetic field). The major difference between these colloids is in the size of the particles in the suspension (\textit{nanoscale} for ferrofluids, and \textit{micro} for magnetorheological fluids), see
\cite{vekasFerrofluidsMagnetorheologicalFluids2008}, which behave as a homogeneous continuum. Such a small diameter of the particles, therefore, calls for a \textit{macroscopic} (or \textit{effective}, or \textit{averaging}, or \textit{homogenized}, or \textit{upscaled}, or \textit{coarse}) description of the given suspension.

The macroscopic response of highly heterogeneous multiscale media is often modelled phenomenologically, which typically does not lead to a straightforward control over the effective properties, that is essential in the case of the ferro- or magnetorheological fluids. However, if the underlying (micro)structure of heterogeneities is periodic, then the \textit{mathematical theory of homogenization}, that yields a homogenized model depending on the microstructure and properties of the constituents, can be employed.  This rigorous approach can thus lead to the design of new materials with the desired properties.

In the framework of a homogenization method, the studies most directly related to this current contribution are \cite{levySuspensionSolidParticles1983,levyHomogenizationMechanicsNondilute1988,nikaMultiscaleModelingMagnetorheological2020}. More specifically, in \cite{levySuspensionSolidParticles1983},
the behavior of non-dilute suspensions of rigid particles in a Newtonian fluid was studied in the case where the magnetic field was neglected, assuming a locally periodic array of particles. The fluid was described by the stationary Navier-Stokes flow equations, and the motion of particles of an arbitrary shape followed the rigid body motion. Using the formal asymptotics expansion procedure, it was obtained that the homogenized medium was given by a viscous fluid, in general anisotropic, predictably depending on the local microstructure. The evolution equations for the microstructure were also obtained. A similar asymptotic study was developed in  \cite{levyHomogenizationMechanicsNondilute1988}, where a suspension of \textit{magnetized} particles in a Newtonian fluid was considered. As in the previous study, rigid particles of an arbitrary shape were coupled with the stationary Navier-Stokes equation of the carrier fluid with coupling between the fluid flow and the magnetic field imposed through the balances of linear and angular momenta equations, rather than in the fluid or solid phase equations. The method of asymptotic expansions employed in \cite{levyHomogenizationMechanicsNondilute1988} resulted in a set of homogenized equations revealing that, even though the fluid where the magnetized particles were suspended was Newtonian, the effective medium was, in general, non-Newtonian. Note that in \cite{levySuspensionSolidParticles1983,levyHomogenizationMechanicsNondilute1988}, the effect of the external magnetic field was imposed as a volume density force acting on each particle and as a surface density force acting on the boundary of each particle, respectively. 
Later, in \cite{nikaMultiscaleModelingMagnetorheological2020}, the formal method of asymptotic expansions was applied to the quasi-static Maxwell equations coupled with the Stokes equations to capture the effective magnetorheological behaviour given by effective viscosity, and three effective magnetic permeabilities, which all depend on the geometry of the suspension, the volume fraction, the original magnetic permeability of particles, the Alfven number, and the distribution of the particles.  A numerical study, based on the obtained homogenized system, was also developed for a suspension of iron particles in a viscous non-conducting fluid to capture the effect of the chain structures present in the microstructure.

As mentioned above, the approaches adopted in previous homogenization contributions \cite{levySuspensionSolidParticles1983,levyHomogenizationMechanicsNondilute1988,nikaMultiscaleModelingMagnetorheological2020} were \textit{formal}. The goal of the present study is to carry out the \textit{rigorous homogenization analysis} for a suspension of magnetized particles in a slow, viscous flow. To that end, we adopt the  \textit{method of two-scale convergence}, see e.g. \cite{nguetsengGeneralConvergenceResult1989,allaireHomogenizationTwoscaleConvergence1992}. Since the analysis for the full coupling of the Navier-Stokes equations with the Maxwell equations becomes increasingly involved - in terms of necessary existence, regularity, and other kinds of results -, this paper deals with a simplified model of \textit{paramagnetic} or \textit{diamagnetic} particles in a viscous fluid with a one-way coupling between them. 
The paramagnetic model is introduced in \cite{neuringerFerrohydrodynamics1964} (see also \cite[Section 7.2]{tongCambridgeLectureNotesa} or \cite[Chapter 13]{zangwillModernElectrodynamics2013}) and assumes that the \textit{magnetic flux density} $\emB$ is linearly proportional to the \textit{magnetic field strength} $\emH$, i.e. there exists a function $\mu(\xx)$ of the spatial variable $\xx$, such that $\emB = \mu (\xx) \emH$, where the function $\mu(\xx)$ is called the \textit{magnetic permeability}. A similar linear relation holds for diamagnetic particles, see \cite{zangwillModernElectrodynamics2013}. 
For the \textit{microscopic} or \textit{fine-scale} description of the coupled system, we have adopted the set of equations derived in \cite{kangDirectSimulationMethod2008,afkhamiFerrofluidsMagneticallyGuided2017}. Thereby, 
the fluid is assumed to be electrically non-conducting, the particles are inertialess, and the contribution from Brownian motions is negligible. The Stokes law governs the motion of the ambient fluid, while the particles exhibit rigid body motion, and the magnetic force, exerted on the particles due to magnetic fields, is represented by the divergence of the Maxwell stress tensor, which acts as a body force added to the momentum balance equation. The above implies a one-way coupled system of hydrodynamic and magnetic interactions, where the magnetic field alters the fluid flow.  Clearly, the present simplified model is an initial step in the methodical investigation of the  homogenized description of hydrodynamic and magnetic coupling, where the two-way coupling as well as other descriptions of the particles, such as nonlinear \textit{ferromagnetic} ones, see \cite{tongCambridgeLectureNotesa,zangwillModernElectrodynamics2013}, are intended to be studied. 

Finally, the main result of this paper consists of the derivation of the homogenized or effective response of the suspension described above and the rigorous justification of the obtained asymptotics.  The novelty of the results of this paper is that, to the best of our knowledge, a rigorous analytical justification for this type of homogenization has not been obtained yet.

This paper is organized as follows. In Section \ref{sec:formulation}, the formulation of the problem under consideration is given, that includes the coupling of rigid body motion of paramagnetic particles with a viscous fluid at the fine (micro) scale. The dimensional analysis is carried out in this section as well. The main results, which include the homogenized equation, the cell problems, and the corrector result, will be given in Section \ref{sec:results}. Our concluding remarks and comments are summarized in Section \ref{sec:conclusions}.

\section{Formulation}
\label{sec:formulation}
Throughout this paper, the scalar-valued quantities, such as the pressure $p$, are written in usual typefaces, while
vector-valued or tensor-valued fields, such as the velocity $\vec{u}$ and the Cauchy stress tensor $\vec{\sigma}$, are
written in bold. A special case that does not follow this rule is the effective tensor, 
which is always written in
normal font.  Sequences are indexed by numeric superscripts ($\phi^i$), while
elements of vectors or tensors are indexed by numeric subscripts ($x_i$). Finally, the Einstein summation convention is used whenever applicable.

Let $Y\coloneqq (0,1)^d$ be the unit cell in $\RR^d$, for $d = 2$ or $3$. The unit cell $Y$ is decomposed into $Y=Y_s\cup Y_f \cup \Gamma$, where $Y_s$, representing the magnetic inclusion, and $Y_f$, representing the fluid domain, are open sets in $\mathbb{R}^d$, and $\Gamma$ is the smooth, closed interface that separates them.

Furthermore, let $i = (i_1, \ldots, i_d) \in \ZZ^d$ be a vector of indices and $\ee_1, \ldots, \ee_d$ be canonical basis of $\RR^d$. For a fixed small $\varepsilon > 0,$ define the dilated sets: 
\begin{align*}
    Y^\varepsilon_i 
    \coloneqq \varepsilon (Y + i),~~
    Y^\varepsilon_{i,s}
    \coloneqq \varepsilon (Y_s + i),~~
    Y^\varepsilon_{i,f}
    \coloneqq \varepsilon (Y_f + i),~~
    \Gamma^\varepsilon_i 
    \coloneqq \partial Y^\varepsilon_{i,s}.
\end{align*}
Let $\nn_i,~\nn_{\Gamma}$ and $\nn_{\partial \Omega}$ be unit normal vectors on $\Gamma^{\varepsilon}_{i}$ pointing outward $Y^\varepsilon_{i,s}$, on $\Gamma$ pointing outward $Y_{s}$ and on $\partial \Omega$ pointing outward, respectively.

We assume that the magnetic permeability is a $Y$-periodic function $\mu \in L^{\infty}(\RR^d)$:
\begin{align*}
    \mu (\zz + m \ee_{k}) = \mu (\zz), \qquad \forall \zz \in \RR^{d}, \quad \forall m \in \ZZ, \quad \forall k \in \{ 1,\ldots, d \}.
\end{align*} 
Let $\Omega \subset \RR^d$ be a simply connected and bounded domain of class $C^3$, so that the effective velocity $\vec{u}^0$ derived below will be from $C^1(\Omega)^d$, as claimed in \cref{regularity-u0}.
We define
\begin{align*}
    I^{\varepsilon} 
    \coloneqq \{ 
    i \in \ZZ^d \colon Y^\varepsilon_i \subset \Omega
    \},~~
    \Omega_s^{\varepsilon} 
    \coloneqq \bigcup_{i\in I^\varepsilon}
Y_{i,s}^{\varepsilon},~~
    \Omega_f^{\varepsilon} 
    \coloneqq \Omega \setminus \Omega_s^{\varepsilon},~~
    \Gamma^\varepsilon 
    \coloneqq \bigcup_{i \in I^\varepsilon} \Gamma^\varepsilon_i.
\end{align*}
Moreover, $\di \us $ and $\di \Gamma^\varepsilon$ denote the surface measure on $\partial\Omega$ and $\Gamma^\varepsilon$, respectively. We assume further that $\Lambda^{-1} \le \mu(\xx) \le \Lambda$, for all $\xx\in \Omega$ and some $\Lambda > 0$. 


Suppose that $\bg \in H^1({\Omega})^d$, which can be regarded as an external force.
As mentioned in the Introduction, the carrier fluid is described by the Stokes equation. To that end, denote by $\eta>0$ and $\rho_f>0$ the fluid viscosity and the fluid density, respectively.  Let $\uu^{\varepsilon}$ and $p^{\varepsilon}$ 
 be the fluid velocity and the fluid pressure, respectively. Also, in a space free of current, the magnetic field strength is given by $\emH^{\varepsilon} = \nabla \varphi^{\varepsilon}$, for some magnetic potential $\varphi^{\varepsilon}(\xx)$. 
We are looking for the functions $\uu^{\varepsilon} \in H_0^1(\Omega)^d$, $p^{\varepsilon}\in L^2(\Omega)/\RR$, and $\varphi^{\varepsilon} \in H^1(\Omega)/\RR$ satisfying the following boundary value problem:
\begin{subequations}
  \label{eq:pn410} 
\begin{align}
\label{eq:pn411}
  -\Div \left[ \vec{\sigma}(\uu^{\varepsilon},p^{\varepsilon}) +
  \bT(\varphi^{\varepsilon}) \right]
  &= \rho_f\, \bg, && \text{ in } \Omega_f^{\varepsilon}\\
  \label{eq:pn419}
  \Div \uu^{\varepsilon}
  &= 0, && \text{ in } \Omega_f^{\varepsilon}\\
  \label{eq:pn420}
 \DD(\uu^{\varepsilon}) 
  &= 0, && \text{ in }\Omega_s^{\varepsilon} \\
  \label{eq:pn421}
  -\Div \left[ \mu \left( \frac{\xx}{\varepsilon} \right) \nabla
  \varphi^{\varepsilon} \right]
  &= 0 && \text{ in }\Omega,
\end{align}
\end{subequations}
 where:
\begin{align*}
\DD(\uu^{\varepsilon}) 
&\coloneqq \frac{\nabla \uu + \nabla^{\top}\uu}{2},\qquad
  \vec{\sigma}(\uu^{\varepsilon},p^{\varepsilon})
  \coloneqq 2 \eta \DD(\uu^{\varepsilon}) - p^{\varepsilon}\II,\\
  \bT(\varphi^{\varepsilon})
  &\coloneqq  \mu \left( \frac{\xx}{\varepsilon} \right) \left( \nabla
    \varphi^{\varepsilon} \otimes \nabla \varphi^{\varepsilon} -
    \frac{1}{2} \abs{\nabla \varphi^{\varepsilon}}^2 \II \right),
\end{align*}
are the \textit{rate of strain}, the \textit{Cauchy stress} and the \textit{Maxwell stress} tensors, respectively.
 We briefly explain the physics behind equations \eqref{eq:pn410}:
 \begin{itemize}[wide]
    \item Equation \eqref{eq:pn411} is the conservation of momentum, and $\bg$ in the right-hand side of the equation can be regarded as the \textit{gravity force}.
     \item Equation \eqref{eq:pn421} carries out information of \emph{no magnetic monopoles} ($\Div \mathbf{B} = 0$) and the \emph{linear constitutive relation} ($\mathbf{B}=\mu \mathbf{H}$).
     \item Equation \eqref{eq:pn420} establishes that $\Omega_s^\varepsilon$ is a \emph{rigid region}, which is equivalent (see \cite[Theorem 3.2]{saramitoComplexFluids2016}) to $\uu^\varepsilon \vert_{Y_{i,s}^\varepsilon} = \mathbf{U}_i + \mathbf{R}_i \times (\xx - \mathbf{C}_i)$, where $\mathbf{C}_i$ is the center of mass of particle $Y_{i,s}^\varepsilon$, and the constant translational velocity $\mathbf{U}_i$ and rotational velocity $\mathbf{R}_i$  are Lagrange multipliers associated to the constraints:
 \end{itemize}
\begin{align}
\label{eq:pn414}
\int_{\Gamma_{i}^{\varepsilon}} \left[
  \vec{\sigma}(\uu^{\varepsilon},p^{\varepsilon}) +
  \bT(\varphi^{\varepsilon}) \right] \nn_i \di \Gamma^{\varepsilon}
  = 0;\quad
  \int_{\Gamma_{i}^{\varepsilon}}  \left[
  \vec{\sigma}(\uu^{\varepsilon},p^{\varepsilon}) +
  \bT(\varphi^{\varepsilon}) \right]\nn_i \times \nn_i \di \Gamma^{\varepsilon}
  = 0.
\end{align}
The constraints above 
are the balance equations for drag forces and torques on the particle boundaries, respectively. 

Finally, for a given \emph{divergence-free} $\kk \in H^1({\Omega})^d$, satisfying the compatibility condition $\int_{\partial\Omega} \kk \cdot \nn_{\partial \Omega} \di \us=0$, we consider the following boundary conditions:
\begin{align}
\label{eq:pn416}
  \uu^{\varepsilon}
  &= 0 \text{ on }\partial \Omega, \qquad \text{ and }\qquad
  \left( \mu \nabla \varphi^{\varepsilon} \right)\cdot \nn_{\partial \Omega}
  = \kk \cdot \nn_{\partial\Omega} \text{ on }\partial \Omega.
\end{align}

To write the variational formulation of problem \eqref{eq:pn410}-\eqref{eq:pn416}, we introduce the following space:
\begin{align}
\label{spaceU}
  \calU^{\varepsilon}
  &\coloneqq  \left\{
    \vv \in H_0^1(\Omega)^d \colon \DD(\vv) = 0 \text{ in }
    \Omega_s^{\varepsilon}, ~\Div \vv = 0 \text{ in }\Omega_f^{\varepsilon}
    \right\}.
\end{align}
It can be shown that the space $\calU^\varepsilon$ is a Hilbert space.
Equation \eqref{eq:pn410} is derived from the variational problem: 
\begin{align*}
    \uu^\varepsilon = \argmin_{\substack{\vv^\varepsilon
    \in\, \calU^\varepsilon\\ \varphi^\varepsilon \in\, \argmin \mathcal{E}_2 } } \mathcal{E}_1 (\vv^\varepsilon), \qquad \text{ where: }
\end{align*}
\begin{align*}
    \mathcal{E}_1 (\uu^\varepsilon)
    &\coloneqq
    \eta \int_{\Omega^\varepsilon_f} \DD(\uu^\varepsilon) : \DD(\uu^\varepsilon) \di \xx -{\int_{\Omega^\varepsilon_f} \rho_f \bg \cdot \uu^\varepsilon \di \xx} + \int_{\Omega^\varepsilon_f} \bT(\varphi^\varepsilon) : \DD(\uu^\varepsilon) \di \xx,\\
    \mathcal{E}_2 (\varphi^\varepsilon)
    &\coloneqq
    \frac{1}{2} \int_\Omega \mu \nabla \varphi^\varepsilon \cdot \nabla \varphi^\varepsilon \di \xx - \int_{\partial\Omega} \left(\kk \cdot \nn_{\partial\Omega}\right) \varphi^\varepsilon \di \xx.
\end{align*}

\subsection*{Dimensional Analysis}
\label{sec:dimensional-analysis}
 
Let $L, U, B, \mu_c$ be our characteristic scales corresponding to
length, velocity, magnetic field and magnetic permeability, respectively. Let
$\xx^{*} \coloneqq \frac{\xx}{L}, \uu^{*} \coloneqq \frac{\uu}{U},
p^{*} \coloneqq \frac{p L}{\eta U}, \bg^{*} \coloneqq \frac{\bg
  L}{U^2}, \mu^{*} \coloneqq \frac{\mu}{\mu_c}$ and
$\varphi^{\varepsilon *}\coloneqq
\frac{\varphi^{\varepsilon}}{B L}$.
The dimensionless quantities that appear are the (hydrodynamic) \emph{Reynolds number}
$\nRe = \frac{\rho_f UL}{\eta}$ and the \emph{Alfven number}
$S = \frac{B^2L}{\eta \mu_c U}.$ 
The non-dimensional version of \eqref{eq:pn411} is: 
\begin{align}
\label{eq:pb410}
 -\Div^{*} \left[ \vec{\sigma}^{*}(\uu^{\varepsilon * },p^{\varepsilon *}) +
  \bT^{*}(\varphi^{\varepsilon *}) \right] = \nRe \bg^{*}.
\end{align}
In the Stokes regime, we have $\nRe \ll 1$, so for simplicity, we assume that the right hand side of \eqref{eq:pb410} vanishes (see also \cref{sec:dimensional-analysis-1}). We
obtain the non-dimensional version of \eqref{eq:pn410} (dropping the
$*$ for clarity of exposition):
\begin{subequations}
  \label{eq:p410} 
\begin{align}
\label{eq:p411}
  -\Div \left[ \vec{\sigma}(\uu^{\varepsilon},p^{\varepsilon}) +
  \bT(\varphi^{\varepsilon}) \right]
  &= 0, && \text{ in } \Omega_f^{\varepsilon}\\
  \label{eq:p419}
  \Div \uu^{\varepsilon}
  &= 0, && \text{ in } \Omega_f^{\varepsilon}\\
  \label{eq:p420}
\DD(\uu^{\varepsilon}) 
  &= 0, &&  \text{ in }\Omega_s^{\varepsilon} \\
  \label{eq:p421}
  -\Div \left[ \mu \left( \frac{\xx}{\varepsilon} \right) \nabla
  \varphi^{\varepsilon} \right]
  &=  0 && \text{ in }\Omega,
\end{align}
\end{subequations}
together with the balance equations:
\begin{align}
\label{eq:p413}
\int_{\Gamma_{i}^{\varepsilon}} \left[
  \vec{\sigma}(\uu^{\varepsilon},p^{\varepsilon}) +
  \bT(\varphi^{\varepsilon}) \right] \nn_i\di \Gamma^{\varepsilon}
  =0=
  \int_{\Gamma_{i}^{\varepsilon}} \left( \left[
  \vec{\sigma}(\uu^{\varepsilon},p^{\varepsilon}) +
  \bT(\varphi^{\varepsilon}) \right]\nn_i \right) \times \nn_i\di \Gamma^{\varepsilon},
\end{align}
and boundary conditions:
\begin{subequations}
  \label{eq:p415} 
\begin{align}
\label{eq:p416}
  \uu^{\varepsilon}
  &= 0, &&\text{ on }\partial \Omega,\\
  \label{eq:p423}
  \left( \mu \nabla \varphi^{\varepsilon} \right)\cdot \nn_{\partial \Omega}
  &= \kk \cdot \nn_{\partial \Omega}
  , &&\text{ on }\partial \Omega,
\end{align}
\end{subequations}
where
\begin{align}
\label{eq:p426}
  \vec{\sigma}(\uu^{\varepsilon},p^{\varepsilon})
  &\coloneqq 2 \DD(\uu^{\varepsilon}) - p^{\varepsilon}\II,\quad
  \bT(\varphi^{\varepsilon})
  \coloneqq S \mu 
  \left( \nabla
    \varphi^{\varepsilon} \otimes \nabla \varphi^{\varepsilon} -
    \frac{1}{2} \abs{\nabla \varphi^{\varepsilon}}^2 \II \right).
\end{align}

\begin{remark}
\label{sec:dimensional-analysis-1}
If the Reynolds number is not small, then one has to keep $\nRe \bg$
on the right hand side of \eqref{eq:pb410}.
During the homogenization process, one would need to consider at least three cases, corresponding to the weak limits of subsequences of
$\mathds{1}_{\Omega^{\varepsilon}_f}$. In some cases, one may encounter a \emph{strange term coming from nowhere}, see
\cite{cioranescuStrangeTermComing1997}. 
\end{remark}

\section{Main Results}
\label{sec:results}

Before formulating the main theorem, we will discuss regularity, which is essential for the existence of solutions. First of all, the
existence of the magnetic potential $\varphi^{\varepsilon}$ is straightforward via the Lax-Milgram~Theorem (see \cref{sec:magn-equat-4}).  However, to prove the existence of the velocity
$\uu^{\varepsilon}$, we need some extra 
technical
assumptions on the regularity of the permeability $\mu$ and the domain
$\Omega$. To that end, we note that $\kk \in H^1(\Omega)^d \subset L^6(\Omega)^d$. Now, from \cite[Theorem
2]{byunConormalDerivativeProblem2005} (see also \cite[Theorem 1.2]{caoGradientEstimatesWeak2019a}), we have the following
regularity result: 
\begin{lemma}
  \label{sec:coupled-equation-6}
 Let $\Omega$ be a given domain in $\mathbb{R}^d$ with a smooth boundary, and suppose $s \in (4,6]$. Then there exists a small number
  $\delta = {\delta (\Lambda, s, \Omega)} > 0$ so that, if
  $ \mathrm{ess} \sup \mu - \mathrm{ess} \inf \mu \le \delta,$
  one has:
\begin{align}
\label{eq:493}
\int_{\Omega} \abs{\nabla \varphi^{\varepsilon}}^{s} \di \xx \le C
  \int_{\Omega} \abs{\kk}^{s} \di \xx,
\end{align}
where the constant $C>0$ is independent of $\varepsilon,~\varphi^{\varepsilon}$ and
{$\kk$}.
\end{lemma}
\begin{remark}
The above result also holds when $\Omega$ is only Lipschitz with a small Lipschitz constant (see \cite{toroDoublingFlatnessGeometry1997} for an estimate).  In the most general setting, we only need to assume $\mu$ is $(\delta, R)$-vanishing and $\Omega$ is $(\delta,R)$-Reifenberg flat; for more details, we refer the reader to \cite{byunConormalDerivativeProblem2005, toroDoublingFlatnessGeometry1997}.
\end{remark}
\begin{remark}
\label{sec:homogenization-2}
Although $s = 4$ is enough for the existence of $\uu^{\varepsilon}$, as one can see in \cref{sec:coupled-equation-2}, the higher regularity, namely $s > 4$, is needed in order to prove the corrector result in
\cref{sec:cell-probl-corr-2}.
\end{remark}

In the sequel, we will use the following functional spaces: 
\begin{itemize}[wide]
    \item $C_{\per}(Y)$ -- the subspace of $C(\RR^d)$ of $Y$-periodic functions;
    \item $C^{\infty}_{\per}(Y)$ -- the subspace of $C^{\infty}(\RR^d)$ of $Y$-periodic functions;
    \item $H^1_{\per}(Y)$ -- the closure of $C^{\infty}_{\per}(Y)$ in the $H^1$-norm;
    \item $\mathcal{D}(\Omega, C^\infty_{\per}(Y))$ -- the space  infinitely differentiable functions from $\Omega$ to $C^\infty_{\per}(Y)$, whose  support is a compact set of $\mathbb{R}^d$ contained in $\Omega$.
   
    \item $L^2_{\per}\left(Y, C(\bar{\Omega})\right)$ -- the space of measurable functions $w \colon \yy \in Y \mapsto w(\cdot,\yy) \in C(\bar{\Omega})$, such that 
    $w$ is periodic with respect to $\yy$ and
    $
    \int_{Y} \left(\sup_{\xx \in \bar{\Omega}} \abs{w(\xx,\yy)}\right)^2 \di \yy 
    < \infty.
    $
    \item $L^p(\Omega, X)$ -- where $X$ is a Banach space and $1 \le p \le \infty$ -- the space of measurable functions $w \colon \xx \in \Omega \mapsto w(\xx) \in X$ such that
    $
    \norm{w}_{L^p(\Omega, X)}
    \coloneqq \left(\int_{\Omega} \norm{w(\xx)}^p_{X} \di \xx\right)^\frac{1}{p} < \infty.
    $
    
\end{itemize}

\begin{definition}
A sequence $\{ v^\varepsilon \}_{\varepsilon>0}$ in $L^2(\Omega)$ is said to \emph{two-scale converge} to $v = v(\xx,\yy)$, with $v \in L^2 (\Omega \times Y)$, and we write $v^\varepsilon \tscale v$, if and only if:
\begin{align}
\label{2sc}
    \lim_{\varepsilon \to 0} \int_\Omega v^\varepsilon(\xx) \psi \left( \xx, \frac{\xx}{\varepsilon}\right) \di \xx 
    = \frac{1}{\abs{Y}} \int_\Omega \int_Y v(\xx,\yy) \psi(\xx,\yy) \di \yy \di \xx,
\end{align}
for any test function $\psi = \psi (\xx, \yy)$ with $\psi \in \mathcal{D}(\Omega, C_\per^\infty (Y))$,
see \cite{allaireHomogenizationTwoscaleConvergence1992,cioranescuIntroductionHomogenization1999,nguetsengGeneralConvergenceResult1989}.
\end{definition}

Note that, by density, if $v^\varepsilon \tscale v$, then \eqref{2sc} holds for any $\psi $ in $ L^2_{\per} \left( Y, C(\bar{\Omega})\right)$ or $L^2\left(\Omega, C_{\per}(Y)\right)$, see e.g. \cite{allaireHomogenizationTwoscaleConvergence1992}. In the sequel, we will use the subscript $\cdot_{\yy}$ to denote the derivative with respect to the second variable $\yy$. We now state our main result in \cref{thm:main}, whose proof will rely on several lemmas discussed over the next sections.
\begin{theorem}[Main theorem]
\label{thm:main}
Suppose that $\mu$ and $\Omega$ satisfy the assumptions in \cref{sec:formulation} and  \cref{sec:coupled-equation-6}. Then, the solution
triple $(\varphi^{\varepsilon}, \uu^{\varepsilon},p^{\varepsilon}) \in (H^1(\Omega)/\RR) \times H_0^1(\Omega)^d \times (L^2(\Omega)/\RR)$
of \eqref{eq:p410} two-scale converges to the unique solution
$(\varphi^0,\uu^0,p^0) \in (H^1(\Omega)/\RR) \times H_0^1(\Omega)^d \times
L^2(\Omega \times Y)/\RR$ of the following ``two-scale homogenized problem":
\begin{subequations}
\label{eq:p444}
  \begin{align}
    \label{eq:p459}
    -\Div \left[ \frac{1}{\abs{Y}}\int_Y \mu(\yy) \left( \nabla \varphi^0 +
    \nabla_{\yy} \varphi^1 \right)\di \yy \right]
    &= {0}, \text{ in } \Omega,\\
    -\Div_{\yy} \left[ \mu(\yy) \left( \nabla \varphi^0 + \nabla_{\yy}
    \varphi^1\right) \right]
    &=0, \text{ in } \Omega \times Y,\\
    \left[\frac{1}{\abs{Y}} \int_Y \mu (\yy) \left( \nabla \varphi^0 +
    \nabla_{\yy}\varphi^1 \right)\di \yy \right]\cdot \nn_{\partial \Omega}
    &={\kk \cdot \nn_{\partial\Omega}}, \text{ on }\partial \Omega,\\
    \label{eq:p476}
    -\Div \left[ \frac{1}{\abs{Y}}\int_Y \left( \vec{\sigma}^0 + \bT^0 \right) \di \yy \right]
    &= 0, \text{ in } \Omega,\\
    \label{eq:p477}
    \Div_{\yy} \left( \vec{\sigma}^0 + \bT^0 \right)
    &=0, \text{ in } \Omega \times Y_f,\\
    \label{eq:494}
    \Div \uu^0
    &= 0, \text{ in }\Omega,\\
    \label{eq:a494}
    \Div_{\yy} \uu^1 
    &= 0, \text{ in }\Omega \times Y,\\
    \label{eq:496}
    \DD(\uu^0) + \DD_{\yy}(\uu^1)
    &= 0, \text{ in } \Omega \times Y_s,\\
  \label{bdr-cell2}
  {\int_{\Omega} \int_{\Gamma} \left(
  \vec{\sigma}^0 +
  \bT^0 \right) \nn_{\Gamma}\di \Gamma_{\yy} \di \xx}
  &= 0,\\ 
  \label{bdr-cell3}
  {\int_{\Omega} \int_{\Gamma} \left[ \left(
  \vec{\sigma}^0 +
  \bT^0 \right)\nn_{\Gamma} \right] \times \nn_{\Gamma}\di \Gamma_{\yy} \di \xx}
  &= 0, 
  \end{align}
\end{subequations}
with the constitutive laws: 
\begin{align}
\label{eq:p475}
  \vec{\sigma}^0
  &\coloneqq
    2 \left[
    \DD(\uu^0) + \DD_{\yy}(\uu^1) \right]-p^0\II ~ \text{ in
    }\Omega,\\
  \label{eq:502}
  \bT^0
  &\coloneqq
    S \mu \left( \left[ \nabla \varphi^0 + \nabla_{\yy} \varphi^1
    \right] \otimes \left[ \nabla \varphi^0 + \nabla_{\yy}\varphi^1
    \right] - \frac{1}{2} \abs{\nabla \varphi^0 +
    \nabla_{\yy}\varphi^1}^2 \II \right) ~ \text{ in }\Omega,
\end{align}
where $\uu^1 \in L^2(\Omega, H_{\per}^1(Y)^d/\RR),~\varphi^1 \in
L^2(\Omega, H^1_{\per}(Y)/\RR)$ are given by the two-scale limits 
$
  \nabla \uu^{\varepsilon}
  \tscale \nabla \uu^0 + \nabla_{\yy}
  \uu^1 \text{ and }
  \nabla \varphi^{\varepsilon}
  \tscale \nabla \varphi^0 + \nabla_{\yy}\varphi^1.
  $
\end{theorem}



We adopt the following lemma several times throughout the paper. 
\begin{lemma}[Averaging Lemma \cite{cioranescuIntroductionHomogenization1999}]
\label{ave-lem}
Let $\phi \in L^p(\Omega; C_{\per}(Y))$ with $1 \le p < \infty$. Then $\phi (\cdot, \cdot/\varepsilon) \in L^p(\Omega)$ with
$ \norm{\phi \left( \cdot, \frac{\cdot}{\varepsilon}\right)}_{L^p(\Omega)} \le \norm{\phi (\cdot, \cdot)}_{L^p(\Omega; C_{\per}(Y)},
$
and: 
\begin{align*}
    \phi \left( \cdot, \frac{\cdot}{\varepsilon} \right) \rightharpoonup \frac{1}{\abs{Y}} \int_{Y} \phi(\cdot,\yy)\di \yy \text{ weakly in } L^p(\Omega).
\end{align*}
\end{lemma}

\subsection{The Magnetostatic Equation}
\label{sec:magn-equat-4}

Multiplying \eqref{eq:p421} by $\tau \in H^1(\Omega)$ and integrating by parts yields: 
\begin{align}
\label{eq:p427}
\int_{\Omega} \left[ \mu \left( \frac{\xx}{\varepsilon} \right) \nabla
  \varphi^{\varepsilon}\right] \cdot \nabla \tau \di \xx
  &= {\int_{\partial \Omega} \left(\kk \cdot \nn_{\partial\Omega}\right) \tau \di \us},
\end{align}
{where the boundary integral on $\partial \Omega$
  vanishes by \eqref{eq:p423}, and the boundary integrals on
  the particle interfaces $\Gamma_i^\epsilon$ vanish by the continuity of the flux
  $\mu\nabla \varphi^\varepsilon \cdot \nn_i$ across $\Gamma^\varepsilon_i$, for $i\in I^{\varepsilon}$}.
 
For $\phi \in H^1(\Omega)/\RR$, we have that $\norm{\phi}_{H^1(\Omega)/\RR} \coloneqq \norm{\nabla \phi}_{L^2(\Omega)}$ is a norm of $H^1(\Omega)/\RR$, by the Poincar\'{e}-Wirtinger inequality.  
Using $H^1({\Omega})^d \subset H (\Div,\Omega)$ and the fact that the field $\kk$ is divergence free, by the Lax--Milgram Theorem, we have that, for each $\varepsilon > 0$,
problem \eqref{eq:p427} admits a unique solution $\varphi^{\varepsilon}$ in $H^1(\Omega)/\RR$
satisfying:
\begin{align}
\label{eq:p429}
  \norm{\varphi^{\varepsilon}}_{H^1(\Omega)/\RR}
  \le C \left(\norm{\kk}_{L^2(\Omega)^d} + \norm{\Div\kk}_{L^2(\Omega)}\right) = C \norm{\kk}_{L^2(\Omega)^d},
\end{align}
where $C = C(\mu)$ is a constant depending only on $\mu$.

\begin{lemma}
\label{sec:magn-equat-2}
Suppose $\varphi^{\varepsilon} \in H^1(\Omega)/\RR$ satisfies \eqref{eq:p421} and
\eqref{eq:p423}.
Then, there exist
$\varphi ^0 \in H^1(\Omega)/\RR$ and $\varphi^1 \in
L^2(\Omega,H_{\per}^1(Y)/\RR)$, such that $\varphi^{\varepsilon}
\tscale \varphi^0,$ $\nabla \varphi^{\varepsilon} \tscale \nabla
\varphi^0 + \nabla_{\yy} \varphi^1$, with: 
\begin{subequations}
  \label{eq:p432} 
\begin{align}
\label{eq:p412}
  -\Div \left[ \frac{1}{\abs{Y}}\int_Y \mu(\yy) \left( \nabla \varphi^0 (\xx) +
  \nabla_{\yy} \varphi^1(\xx,\yy) \right) \di \yy \right]
  &= {0},
  &&\text{ in } \Omega,\\
\label{eq:p433}
  -\Div_{\yy} \left[ \mu (\yy) (\nabla \varphi^0(\xx) + \nabla_{\yy} \varphi^1(\xx,\yy)) \right]
  &= 0 &&\text{ in } \Omega\times Y,\\
  \label{eq:p432c}
  \left[ \frac{1}{\abs{Y}}\int_Y \mu (\yy) \left( \nabla \varphi^0 +
    \nabla_{\yy}\varphi^1 \right)\di \yy \right]\cdot \nn_{\partial \Omega}
    &=\kk \cdot \nn_{\partial\Omega}, &&\text{ on }\partial \Omega.
\end{align}
\end{subequations}

\end{lemma}

\begin{proof}
\begin{enumerate}[wide]
\item By \eqref{eq:p429}, there exist $\varphi^0 \in H^1(\Omega)/\RR$ and
  $\varphi^1 \in L^2(\Omega, H_{\per}^1(Y)/\RR)$, such that (up to a subsequence)
  $ 
  \varphi^{\varepsilon}
  \tscale \varphi^0$
  and
  $
  \nabla \varphi^{\varepsilon}
  \tscale \nabla \varphi^0 + \nabla_{\yy} \varphi^1 $ (see \cite{allaireHomogenizationTwoscaleConvergence1992}).
In \eqref{eq:p427}, let $\tau(\cdot)=\tau^0 (\cdot) +
  \varepsilon\tau^1 \left( \cdot,\frac{\cdot}{\varepsilon}
  \right)$, with $\tau^0 \in \mathcal{D}(\Omega)$ and $\tau^1
  \in \mathcal{D} (\Omega, C^{\infty}_{\per}(Y))$, we obtain:
\begin{align*}
  \begin{split}
 \int_{\Omega} \mu \left( \frac{\xx}{\varepsilon} \right) \nabla
  \varphi^{\varepsilon}(\xx) \cdot\left[ \nabla \tau^0(\xx)+ \varepsilon \nabla_{\xx}
  \tau^1 \left( \xx, \frac{\xx}{\varepsilon} \right) +
  \nabla_{\yy}\tau^1 \left( \xx, \frac{\xx}{\varepsilon} \right) \right]
  \di \xx =0.
  \end{split}
\end{align*}
Now, using
$\mu \left( \yy \right) \left( \nabla \tau^0(\xx)
  + \nabla_{\yy}\tau^1(\xx,\yy) \right) \in L_{\per}^2(Y, C(\bar{\Omega}))$
as a test function in the two-scale convergence of
$\nabla \varphi^{\varepsilon}$, as $\varepsilon \to 0$, we obtain:
\begin{align}
\label{eq:pp324}
  \begin{split}
\frac{1}{\abs{Y}}\int_{\Omega} \int_Y \mu(\yy) \left( \nabla \varphi^0(\xx) + \nabla_{\yy} \varphi^1(\xx,\yy)
  \right)\cdot\left( \nabla \tau^0(\xx) + \nabla_{\yy} \tau^1(\xx,\yy) \right)
  \di \yy \di \xx =0.
  \end{split}
\end{align}
In \eqref{eq:pp324}, choosing $\tau^0(\xx) \equiv 0$, $\tau^1(\xx,\yy) \equiv \tau^1(\yy)$, and integrating by parts, we obtain \eqref{eq:p433}. Similarly, choosing $\tau^1(\xx,\yy) \equiv 0$ leads to \eqref{eq:p412}.

\item Letting $\varepsilon \to 0$ in \eqref{eq:p427}, integrating by parts, and using \eqref{eq:p412}, we obtain:
\begin{align*}
    \int_{\partial \Omega} &\left(\kk(\xx) \cdot \nn_{\partial\Omega}\right) \tau(\xx) \di \us
    = \frac{1}{\abs{Y}}\int_{\Omega} \int_Y \mu(\yy) \left( \nabla \varphi^0(\xx) + \nabla_{\yy} \varphi^1(\xx,\yy)
  \right)\cdot \nabla \tau(\xx)
  \di \yy \di \xx\\
  &=
  \frac{1}{\abs{Y}}\int_{\partial\Omega} \tau (\xx) \left[ 
  \int_Y \mu(\yy) \left( \nabla \varphi^0(\xx) + \nabla_{\yy} \varphi^1(\xx,\yy)
  \right)  \di \yy \right]\cdot\nn_{\partial\Omega} \di \us,
\end{align*}
which implies \eqref{eq:p432c}.
\item 
 For $ (\tau^0, \tau^1)  \in  (H^1(\Omega)/\RR) \times L^2(\Omega,
 H^1_{\per}(Y)/\RR)$, let
$ 
  \norm{(\tau^0,\tau^1)}^2
  \coloneqq \norm{\nabla\tau^0}^2_{L^2(\Omega)}
  + \norm{\tau^1}^2_{L^2(\Omega;
    H_{\per}^1(Y)/\RR)}.
$ 
Using this norm, we apply the Lax-Milgram theorem to the variational problem:
\begin{align}
\label{eq:bp452}
\begin{split}
    \frac{1}{\abs{Y}}\int_{\Omega} \int_Y \mu(\yy) \left( \nabla \varphi^0(\xx) + \nabla_{\yy} \varphi^1(\xx,\yy)
  \right)\cdot\left( \nabla \tau^0(\xx) + \nabla_{\yy} \tau^1(\xx,\yy) \right)
  \di \yy \di \xx\\
  =\int_{\partial \Omega} \left(\kk(\xx) \cdot \nn_{\partial\Omega}\right) \tau^0(\xx)\di \us,
  \end{split}
\end{align}
to obtain that \eqref{eq:p432} has a unique
solution $(\varphi^0, \varphi^1)$ in $(H^1(\Omega)/\RR) \times L^2(\Omega,
 H^1_{\per}(Y)/\RR)$.
\end{enumerate}
\end{proof}

The linearity of \eqref{eq:p433} implies that we can separate the slow
and the fast variable in $\varphi^1(\xx,\yy)$ by:
\begin{align}
\label{eq:p452}
  \varphi^1(\xx,\yy)
  = \omega^i(\yy) \frac{\partial \varphi^0}{\partial x_i}(\xx).
\end{align}
Substituting back in \eqref{eq:p433}, we deduce that $\omega^i \in H^1_{\per}(Y)/\RR,\, 1 \le i
\le d,$ satisfies:
 \begin{align}
\label{eq:p453}
-\Div_{\yy} \left[ \mu(\yy) \left(\ee^i +\nabla_{\yy} \omega^i(\yy)
  \right)  \right]
  &= 0 \text{ in } Y.
\end{align}

Plugging \eqref{eq:p452} into \eqref{eq:p412} and \eqref{eq:p432c}, we obtain:
\begin{align}
\label{eq:p455}
    -\Div \left( \mu^{\eff} \nabla\varphi^0 \right) 
    = {0} \text{ in }\Omega, \quad\text{ and } \quad
    \left( \mu^{\eff} \nabla\varphi^0 \right) \cdot \nn_{\partial\Omega}
    = \kk \cdot \nn_{\partial\Omega} \text{ on }\partial\Omega,
\end{align}
where the effective magnetic permeability is given by: 
\begin{align}
    \label{eq:bp456}
    \mu^\eff_{jk} = \frac{1}{\abs{Y}}\int_{Y} \mu (\yy) \left( \delta_{kj} + \frac{\partial \omega^k}{\partial y_j} \right) \di \yy 
    = \frac{1}{\abs{Y}} \int_Y \mu(\yy) (\ee^k+ \nabla \omega^k(\yy))\cdot \ee^j  \di \yy.
\end{align}

\begin{lemma}
\label{regularity-phi0}
The coefficients of the effective matrix $\mu^{\eff}$ can be written as:
\begin{align}
\label{eq:p456}
  \mu^{\eff}_{jk}
  \coloneqq \frac{1}{\abs{Y}} \int_Y \mu(\yy) (\ee^k+ \nabla \omega^k(\yy))\cdot
  (\ee^j + \nabla \omega^j(\yy)) \di \yy.
\end{align}
Therefore, $\mu^{\eff}$ is symmetric and positive definite. As a consequence, \eqref{eq:p455} has a unique solution (up to a constant) $\varphi^0 \in W^{2,s}(\Omega) \subset C^1(\bar{\Omega})$ where $s > 4$.
\end{lemma}

\begin{proof}
\begin{enumerate}[wide]
\item Testing \eqref{eq:p453} against $\omega^k \in H^1_{\per}/\RR$, one gets
$ 
    \int_Y \mu (\yy) (\ee^j + \nabla_{\yy} \omega^j) \cdot \nabla \omega^k \di \yy = 0,  
$ 
where the boundary term vanishes due to periodicity.
{From the above identity and \eqref{eq:bp456},} we obtain \eqref{eq:p456}.
Now, for $\vec{\zeta} \in \RR^d$,  we have: 
\begin{align*}
    \mu^{\eff} \vec{\zeta} \cdot \vec{\zeta} 
    = \mu^{\eff}_{jk} \zeta_j \zeta_k
    &= \frac{1}{\abs{Y}} \int_Y \mu(\yy) (\ee^k+ \nabla \omega^k(\yy))\zeta_k\cdot (\ee^j + \nabla \omega^j(\yy))\zeta_j \di \yy\\
  &\geq\frac{1}{\Lambda\abs{Y}}\abs{ \zeta_k\int_Y (\ee^k+ \nabla \omega^k(\yy))\di \yy}^2
    =\frac{\abs{Y}}{\Lambda} \abs{\vec{\zeta}}^2,
\end{align*}
because of Jensen's inequality and $\int_Y (\ee^k+ \nabla \omega^k(\yy))\di \yy = \abs{Y} \delta_{kj} \ee^j$.

\item Since $\Omega$ is of class $C^3$ and $\kk \cdot \nn \in H^{1/2}(\partial \Omega)$ because $\kk \in H^1(\Omega)^d$, we have that \eqref{eq:p455} admits a solution $\varphi^0 \in H^2(\Omega)$ by \cite[Theorem 5.50]{demengelFunctionalSpacesTheory2012}. 
On the other hand, 
$\nabla \varphi^\varepsilon  \tscale \nabla \varphi^0 + \nabla_{\yy} \varphi^1$ implies $\nabla \varphi^\varepsilon \rightharpoonup \frac{1}{\abs{Y}} \int_Y (\nabla \varphi^0 + \nabla_{\yy} \varphi^1) \di \yy = \nabla\varphi^0$ in $L^2(\Omega)$
(here we use \eqref{eq:p452}). The latter, together with \eqref{eq:493}, implies that $\nabla \varphi^\varepsilon$ converges weakly in $L^s(\Omega)^d$ to $\nabla\varphi^0$. Moreover, since $s > 4$, we have $\varphi^0 \in W^{2,s}(\Omega) \subset C^1(\bar{\Omega})$.
\end{enumerate}
\end{proof}

\begin{lemma}[First order corrector result for the magnetic potential]
  \label{sec:magn-equat-3}
Let $\varphi^{\varepsilon},\, \varphi^0$ and $\varphi^1$ be as in
\cref{sec:magn-equat-2}, then: 
\begin{align}
\label{eq:p463}
\lim_{\varepsilon\to 0}\norm{\nabla \varphi^{\varepsilon} (\cdot) - \nabla \varphi^0(\cdot) -
  \nabla_{\yy} \varphi^1 \left( \cdot, \frac{\cdot}{\varepsilon}
  \right)}_{L^2(\Omega)} =0.
\end{align}
\end{lemma}
\begin{proof}
\begin{enumerate}[wide]
\item 
Recall $\mu \ge \Lambda^{-1} > 0$, so:
\begin{align}
\label{eq:p457}
  \begin{split}
   &\Lambda^{-1}
    \norm{\nabla \varphi^{\varepsilon} (\cdot) - \nabla
      \varphi^0(\cdot) - \nabla_{\yy} \varphi^1 \left( \cdot,
        \frac{\cdot}{\varepsilon} \right)}^2_{L^2(\Omega)}\\
  &\le \int_{\Omega} \mu \left( \frac{\xx}{\varepsilon} \right) \nabla
  \varphi^{\varepsilon}(\xx) \cdot \nabla \varphi^{\varepsilon}(\xx)
  \di \xx \\
  &\qquad - 2 \int_{\Omega} \mu \left( \frac{\xx}{\varepsilon} \right)
  \left( \nabla \varphi^0(\xx) +  \nabla_{\yy} \varphi^1 \left( \xx,
      \frac{\xx}{\varepsilon} \right) \right) \cdot \nabla
  \varphi^{\varepsilon} (\xx) \di \xx\\
  &\qquad + \int_{\Omega} \mu \left( \frac{\xx}{\varepsilon} \right)
  \left( \nabla \varphi^0(\xx) + \nabla_{\yy}\varphi^1 \left( \xx,
      \frac{\xx}{\varepsilon} \right) \right) \cdot \left( \nabla
    \varphi^0 (\xx) + \nabla_{\yy} \varphi^1 \left( \xx,
      \frac{\xx}{\varepsilon} \right) \right) \di \xx\\
  &\eqqcolon
  \mathcal{I}_1 + \mathcal{I}_2 + \mathcal{I}_3.
  \end{split}
\end{align}

\item 
From \eqref{eq:p427} and by taking $\left( \tau^0, \tau^1 \right) = \left( \varphi^0,\varphi^1 \right)$ in \eqref{eq:bp452}, we obtain:
\begin{align}
\label{eq:p458}
    \lim_{\varepsilon \to 0} \mathcal{I}_1
    &= \lim_{\varepsilon \to 0} \int_{\partial\Omega} (\kk \cdot \nn_{\partial\Omega})
    \varphi^{\varepsilon} \di \us 
    = \lim_{\varepsilon \to 0} \int_{\Omega} \Div(\varphi^{\varepsilon} \kk) \di \xx
    = \lim_{\varepsilon \to 0} \int_{\Omega} \kk \cdot \nabla \varphi^\varepsilon \di \xx\\ \nonumber
    &= \frac{1}{\abs{Y}} \int_{\Omega} \int_{Y} \kk \cdot (\nabla\varphi^0 + \nabla_{\yy} \varphi^1) \di \yy \di \xx
    = \int_{\Omega} \kk \cdot \nabla \varphi^0 \di \xx
    = \int_{\partial\Omega} (\kk \cdot \nn_{\partial\Omega})
    \varphi^{0} \di \us\\ \nonumber
    &= \frac{1}{\abs{Y}} \int_{\Omega} \int_Y \mu (\yy) \left( \nabla
      \varphi^0(\xx)+ \nabla_{\yy}\varphi^1(\xx,\yy) \right) \cdot
    \left( \nabla \varphi^0(\xx)+ \nabla_{\yy} \varphi^1(\xx,\yy)
    \right) \di \yy \di \xx. 
\end{align}
where, in the second to last identity, we use the fact that $\varphi^0 \in C^1(\bar{\Omega})$ (\cref{regularity-phi0}).

\item 
To deal with $\mathcal{I}_2$, we first observe that, by \eqref{eq:p452}, we have:
\begin{align*}
    \mu \left( \frac{\xx}{\varepsilon} \right)\left( \nabla
      \varphi^0+ \nabla_{\yy} \varphi^1 \left( \xx,
        \frac{\xx}{\varepsilon} \right) \right)\cdot \nabla
    \varphi^{\varepsilon}
    = \mu \left( \frac{\xx}{\varepsilon} \right) \left[
      \frac{\partial\varphi^0}{\partial x_i} \left( \ee^i + \nabla_{\yy}
        \omega^i \left( \frac{\xx}{\varepsilon} \right) \right)
    \right] \cdot \nabla \varphi^{\varepsilon}.
\end{align*}
By \cref{regularity-phi0} and \eqref{eq:p453}, $\varphi^0 \in C^1(\bar{\Omega})$  and $\omega^i \in L^2_{\per}(Y)$. Therefore, we can regard $\mu \left(\yy \right) \left[
  \frac{\partial \varphi^0}{\partial x_i}(\xx) \left( \ee^i +
    \nabla_{\yy} \omega^i \left( \yy \right)
  \right) \right] \in L^2_{\per}(Y, C(\bar{\Omega}))$ as a test
function when evaluating the limit of $\mathcal{I}_2$ as $\varepsilon\to0$: 
\begin{align}
\label{eq:p460}
  \begin{split}
    \lim_{\varepsilon \to 0} \mathcal{I}_2
    &= - \frac{2}{\abs{Y}} \int_{\Omega}\int_Y \mu (\yy) \left( \nabla
    \varphi^0(\xx) + \nabla_{\yy}\varphi^1(\xx,\yy) \right)\cdot\left( \nabla
    \varphi^0(\xx) + \nabla_{\yy}\varphi^1(\xx,\yy) \right)\di \yy \di \xx.
  \end{split}
\end{align}

\item 
Finally, for the integral $\mathcal{I}_3$, we use \eqref{eq:p452} to
write: 
\begin{align*}
  \begin{split}
\mu \left( \frac{\xx}{\varepsilon} \right) \left( \nabla \varphi^0
  (\xx) + \nabla_{\yy} \varphi^1 \left( \xx, \frac{\xx}{\varepsilon}
  \right) \right) \cdot \left( \nabla \varphi^0 (\xx) + \nabla_{\yy}
  \varphi^1 \left(\xx, \frac{\xx}{\varepsilon}\right) \right)
\\
= \mu \left( \frac{\xx}{\varepsilon} \right) \frac{\partial
  \varphi^0}{\partial x_i}(\xx) \frac{\partial \varphi^0}{\partial x_j}
  (\xx) \left( \ee^i + \nabla_{\yy}\omega^i \left(
  \frac{\xx}{\varepsilon} \right) \right) \cdot \left( \ee^j +
  \nabla_{\yy} \omega^j \left( \frac{\xx}{\varepsilon} \right)
\right).
  \end{split}
\end{align*}
Regarding $\frac{\partial \varphi^0}{\partial x_i}(\xx) \frac{\partial
  \varphi^0}{\partial x_j}(\xx)$ as a test function, by \cref{ave-lem}
we have: 
\begin{align}
\label{eq:p462}
  \begin{split}
    \lim_{\varepsilon \to 0} \mathcal{I}_3
  &= \frac{1}{\abs{Y}} \int_{\Omega} \int_Y \mu(\yy) \left( \nabla
    \varphi^0(\xx) + \nabla_{\yy} \varphi^1(\xx,\yy) \right) \cdot
  \left( \nabla\varphi^0(\xx) + \nabla_{\yy} \varphi^1(\xx,\yy)
  \right) \di \yy \di \xx.
  \end{split}
\end{align}
To that end, the result in \eqref{eq:p463} follows from \eqref{eq:p457}, \eqref{eq:p458}, \eqref{eq:p460} and
\eqref{eq:p462}.
\end{enumerate}
\end{proof}

 At this point, we can conclude from \eqref{eq:p463} and H\"{o}lder's inequality that:
\begin{align}
\label{eq:p466}
\lim_{\varepsilon \to 0}\norm{\bT(\varphi^{\varepsilon})(\cdot) -
  \bT^{0}\left(\cdot, \frac{\cdot}{\varepsilon}\right)}_{L^1(\Omega)^{d \times d}} = 0,
\end{align}
where:
\begin{align*}
    \bT^{0}(\xx,\yy)
    &\coloneqq S \mu \left( \yy \right)
    \left(\vphantom{\frac{1}{2}} \left[ \nabla \varphi^0(\xx) +
        \nabla_{\yy} \varphi^1(\xx,\yy) \right] \otimes \left[ \nabla
        \varphi^0(\xx) + \nabla_{\yy} \varphi^1(\xx,\yy) \right]\right.\\
    &\qquad- \left.  \frac{1}{2} \abs{\nabla \varphi^0(\xx) +
        \nabla_{\yy} \varphi^1(\xx,\yy)}^2 \II \right).
\end{align*}

\subsection{The Coupled Conservation of Momentum Equation}
\label{sec:coupled-equation-2}

In this section, we establish the existence and derive a priori estimates for the velocity $\uu^\varepsilon$ and the pressure $p^\varepsilon$ in \eqref{eq:pn410}. We prove the existence of the solution $\uu^\varepsilon$ of \eqref{eq:p411} in $\calU^\varepsilon$ \eqref{spaceU}. Multiplying \eqref{eq:p411} by $\vv\in\calU^{\varepsilon}$ and integrating by parts, we deduce that: 
\begin{align}
\label{eq:p436}
  2 \int_{\Omega} \DD(\uu^{\varepsilon}):\DD(\vv) \di \xx
  &=  -
    \int_{\Omega} \bT (\varphi^{\varepsilon}):\DD(\vv) \di \xx,
\end{align}
where the boundary terms vanish by \eqref{eq:p415} and \eqref{eq:p413}, and $\charac_{\Omega_f^{\varepsilon}}$ is the characteristic function of the domain $\Omega_f^{\varepsilon}$.
Define: 
\begin{align*}
  a(\uu,\vv)
  \coloneqq 2 \int_{\Omega} \DD(\uu) : \DD(\vv) \di \xx,
\end{align*}
which is a bilinear form on $\calU^{\varepsilon}$.  Note that $a(\cdot,\cdot)$ is continuous and coercive on $\calU^{\varepsilon}$.
\medskip

Assuming that $\mu$ and $\Omega$ satisfy the regularity
assumptions in \cref{sec:coupled-equation-6}, we have that $\bT (\varphi^{\varepsilon})\in L^2(\Omega)^{d \times d}$. Now, define the linear form $\ell$
on $\calU^{\varepsilon}$ by:
\begin{align*}
\ell(\vv) \coloneqq  -
    \int_{\Omega} \bT (\varphi^{\varepsilon}):\DD(\vv) \di \xx.
\end{align*}
Observe that $\ell(\cdot)$ is continuous on $\calU^{\varepsilon}$. The Lax--Milgram theorem can be applied to the variational problem given by: 
\begin{align}
\label{eq:p353}
\text{Find $\uu^{\varepsilon}\in \calU^{\varepsilon}$ such that }
  a(\uu^{\varepsilon},\ww) = \ell(\ww), \text{ for all } \ww \in \calU^{\varepsilon},
\end{align}
to show that it has a unique solution $\uu^{\varepsilon} \in \calU^{\varepsilon}$ that
satisfies: 
\begin{align}
\label{eq:p437}
\begin{split}
  \norm{\uu^{\varepsilon}}_{H^1_0(\Omega)^d}
  &\le C \left( \int_{\Omega} \abs{\kk}^s \di \xx \right)^\frac{1}{2},
  \end{split}
\end{align}
where the last estimate follows from \cref{sec:coupled-equation-6}.

We introduce the spaces:
\begin{align}
\label{defVePe}
\begin{split}
  \calV^{\varepsilon}
  &\coloneqq \left\{ \vv \in H_0^1(\Omega)^d\colon \DD(\vv) = 0 \text{
  in } \Omega_s^{\varepsilon} \right\},\\
  \calP^{\varepsilon}
  &\coloneqq \Div(\calV^{\varepsilon})=\left\{ q \in L_0^2(\Omega)\colon \exists \vv \in
    \calV^{\varepsilon} \text{ such that } q = \Div \vv \right\},
\end{split}
\end{align}
where 
$L_0^2(\Omega)
  \coloneqq \left\{ q \in L^2(\Omega)\colon \int_{\Omega} q \di \xx = 0\right\}.$
It can be shown that $\calV^{\varepsilon}$ and $\calP^{\varepsilon}$ are Hilbert spaces with respect to the $H_0^1-$ and $L^2-$inner products.
The following lemma deals with the existence of the pressure.

\begin{lemma}
\label{sec:coupled-equation-4}
There exists $p^{\varepsilon} \in \calP^{\varepsilon}$, which arises as a Lagrange multiplier of
the minimization problem given by:
\begin{equation}
  \label{eq:p357}
  \vv = \argmin_{\ww \in \calV^{\varepsilon}} J(\ww), \qquad
  \text{subject to } G(\ww) = 0,
\end{equation}
where 
 $J(\vv)\coloneqq \frac{1}{2}a(\vv,\vv) - \ell (\vv),$ and
 $G(\vv)\coloneqq \Div \vv.$ 
The function $p^{\varepsilon}$ is called the pressure, and satisfies: 
\begin{align}
  \label{eq:p438}
  2 \int_{\Omega} \DD(\uu^{\varepsilon}):\DD(\vv) \di \xx
  -\int_{\Omega} p^{\varepsilon} \Div \vv \di \xx
  &= 
    -\int_{\Omega} \bT (\varphi^{\varepsilon}):\DD(\vv) \di \xx, \text{
      for } \vv \in \calV^{\varepsilon}.
\end{align}
\end{lemma}

\begin{proof}
\begin{enumerate}[wide]
\item 
Since the bilinear form $a(\cdot, \cdot)$ is symmetric, the variational problem
\eqref{eq:p353} is equivalent to the minimization problem \eqref{eq:p357}.
Thus, the existence of the unique solution $\uu^{\varepsilon}$ of the variational
problem \eqref{eq:p353} implies that $\uu^{\varepsilon}$ is also the
unique solution of the minimization problem \eqref{eq:p357}.

\item 
  Fix $\uu \in \calV^{\varepsilon}$ and let
  $\ww \in \calV^{\varepsilon}$. Then
  $G(\uu+\ww) - G(\uu) = \Div \ww$. Since
  $\norm{\Div \ww}_{L^2(\Omega)} \le \norm{\nabla \ww}_{L^2(\Omega)^{d\times d}}
  \equiv \norm{\ww}_{H_0^1(\Omega)^d}$, the Fr\'{e}chet derivative of $G\colon \calV^{\varepsilon} \to \calP^{\varepsilon}$
  at $\uu$ is given by:
$G'(\uu) \colon \calV^{\varepsilon} \to \calP^{\varepsilon},~  
\vv \mapsto \Div \vv.$
By construction, $G'(\uu)$ is surjective. Moreover, it is clear that
$G' \colon \calV^{\varepsilon} \to L(\calV^{\varepsilon},
\calP^{\varepsilon})$ is continuous, where $L(\calV^{\varepsilon},
\calP^{\varepsilon})$ is the spaces of bounded linear mappings from $\calV^{\varepsilon}$ to $\calP^{\varepsilon}$. Therefore, $G$ is $C^1$.

\item
  Since $J$ is the sum of a bounded bilinear function and a bounded linear function, then it is $C^1$.  By the Lagrange Multiplier~theorem {\cite[Section
    4.14]{zeidlerAppliedFunctionalAnalysis1995}} and the fact that $\calP^\varepsilon$ is a Hilbert space, there exists
  $p^{\varepsilon} \in \left( \calP^{\varepsilon} \right)^{*} \cong \calP^\varepsilon$, where $\cong$ means isomorphic, such that:
$
a(\uu^{\varepsilon}, \cdot) - \ell(\cdot) - p^{\varepsilon} \Div (\cdot)
  = 0 \text{ in } \calV^{\varepsilon}.
$ 
That is, $(\uu^{\varepsilon},p^{\varepsilon})$ satisfies \eqref{eq:p438}. 
\end{enumerate}
\end{proof}

Now, in \eqref{eq:p438}, we consider $\vv \in H_0^1(\Omega_f^\varepsilon)$, which can be extended by zero to
$\Omega_s^{\varepsilon}$ so that it belongs to $\calV^{\varepsilon}$,  
and integrate by parts over $\Omega_f^{\varepsilon}$, to obtain:
\begin{align*}
  \begin{split}
&\left\langle \nabla p^{\varepsilon},\vv
  \right\rangle_{H^{-1}(\Omega_f^{\varepsilon}),H_0^1(\Omega_f^{\varepsilon})^d}
  = \int_{\Omega_f^{\varepsilon}} \nabla p^{\varepsilon}\cdot \vv \di\xx
    = -\int_{\Omega_f^{\varepsilon}} p^{\varepsilon} \Div \vv \di \xx \\
  & \qquad=  -
    \int_{\Omega_f^{\varepsilon}} \bT (\varphi^{\varepsilon}):\DD(\vv) \di \xx
   -2\int_{\Omega_f^{\varepsilon}}
    \DD(\uu^{\varepsilon}) : \DD(\vv) \di \xx.
  \end{split}
\end{align*}
The above implies that:
\begin{align}
  \label{eq:p345}
  \begin{split}
  \norm{\nabla p^{\varepsilon}}_{H^{-1}(\Omega_f^{\varepsilon})}
    &\le C \left( \int_{\Omega} \abs{\kk}^s \di \xx \right)^\frac{1}{2},
  \end{split}
\end{align}
where the last inequality follows from \eqref{eq:p437} and \eqref{eq:493}.

The next corollary follows from the fact that   ${\norm{p^\varepsilon}_{L_0^2(\Omega_f^{\varepsilon})} \le \norm{\nabla
    p^\varepsilon}_{H^{-1}(\Omega_f^{\varepsilon})}}$  (see \cite[Lemma
IV.1.9]{boyerMathematicalToolsStudy2013a}) and  $p^{\varepsilon} = 0$ in
$\Omega_s^{\varepsilon}$, together with \eqref{eq:p345} and \eqref{eq:p437}.
\begin{corollary}
\label{sec:coupled-equation-5}
For each $\varepsilon > 0,$ there exists a unique solution
$(\uu^{\varepsilon},p^{\varepsilon}) $ of \eqref{eq:p410} in $ {H_0^1(\Omega)^d \times
L_0^2(\Omega)}$ satisfying the a priori estimate: 
\begin{align*}
  \norm{\uu^{\varepsilon}}_{H_0^1(\Omega)^d} + \norm{p^{\varepsilon}}_{L_0^2(\Omega)}
  &\le C  \left( \int_{\Omega} \abs{\kk}^s \di \xx \right)^\frac{1}{2},
\end{align*}
where the constant $C$ does not depend on $\varepsilon.$ Moreover,
$p^{\varepsilon} = 0$ on $\Omega_s^{\varepsilon}$.
\end{corollary}

\subsection{Proof of Theorem \ref{thm:main}}
\label{sec:proof}
Finally, we are ready to prove the main theorem of this paper, whose key steps are assembled from results presented above.
\begin{proof}[Proof]
\begin{enumerate}[wide]
\item The estimates in \cref{sec:coupled-equation-5} and
  \cref{sec:coupled-equation-6} imply that there exist 
  $\uu^0 \in H_0^1(\Omega)^d$,  $\uu^1 \in L^2(\Omega,
  H_{\per}^1(Y)^d/\RR)$, and $p^0 \in L^2(\Omega \times Y)/\RR$ satisfying (up to a subsequence):
\begin{align} \label{eq:p439} 
  \uu^{\varepsilon}
  \tscale \uu^0,\qquad
  \nabla \uu^{\varepsilon}
  \tscale \nabla \uu^0 + \nabla_{\yy} \uu^1,\qquad
  p^{\varepsilon}
  \tscale p^0.
\end{align}

\item
  Let $\psi \in H_0^1(\Omega)$ and regard $H_0^1(\Omega)$ as a
  subspace of $L^2(\Omega \times Y)$. We deduce from \eqref{eq:p419},
  \eqref{eq:p416} and 
  \eqref{eq:p439} that:
\begin{align*}
  \begin{split}
  0 &= \lim_{\varepsilon \to 0} \int_{\Omega} \psi \Div
  \uu^{\varepsilon} \di \xx = -\lim_{\varepsilon \to 0}\int_{\Omega}
  \uu^{\varepsilon} \cdot \nabla \psi\di \xx 
  = -\frac{1}{\abs{Y}} \int_{\Omega} \int_Y \uu^0 \cdot\nabla \psi
  \di \yy \di \xx \\
  &= - \int_{\partial \Omega} \psi \uu^0 \cdot\nn_{\partial\Omega} \di \Gamma (\xx) +
  \int_{\Omega} \psi \Div \uu^0 \di \xx
  = \int_{\Omega} \psi \Div \uu^0 \di \xx.
  \end{split}
\end{align*}
Therefore, we obtain that
$\Div \uu^0 = 0 \text{ in } \Omega.$
Moreover,
$0 = \Div \uu^{\varepsilon} = \mathrm{tr}(\nabla \uu^{\varepsilon}) \tscale
  \mathrm{tr}(\nabla \uu^0 + \nabla_{\yy} \uu^1) = \Div \uu^0 + \Div_{\yy} \uu^1, 
$
so 
$\Div_{\yy} \uu^1 = 0 \text{ in }\Omega \times Y$.
\medskip

\item 
Let $\vec{\Psi} \in \mathcal{D}\left(\Omega,
  C_{\per}^{\infty}(Y)^{d\times d}\right)$ be supported in $\Omega \times
  Y_s$. From \eqref{eq:p420} and 
  \eqref{eq:p439},
  we have:
$$ 
0 = \lim_{\varepsilon \to 0} { \int_{\Omega}
  \DD(\uu^{\varepsilon}):\vec{\Psi} \left(\xx,\frac{\xx}{\varepsilon}\right)} \di \xx = \frac{1}{\abs{Y}}
  \int_{\Omega}\int_{Y_s} \left[\DD(\uu^0) + \DD_{\yy}(\uu^1)\right] :
  \vec{\Psi}{(\xx,\yy)} \di\yy \di \xx,
$$
which implies that 
$ \DD(\uu^0) + \DD_{\yy}(\uu^1) = 0, \text{ on } \Omega \times Y_s.$
\medskip

\item \label{itemDivy}
In \eqref{eq:p438}, by letting $\vv(\xx) = \vec{\psi} \left( \xx,
  \frac{\xx}{\varepsilon} \right)$, for $\vec{\psi} \in
\mathcal{D}(\Omega, C_{\per}^{\infty}(Y)^d)$ supported in {$ \Omega \times Y_f$}, we get:
\begin{align*}
  \begin{split}
    2 \int_{\Omega} \DD(\uu^{\varepsilon}): \left[
      \DD_{\xx}(\vec{\psi}) +
      \frac{1}{\varepsilon}\DD_{\yy}(\vec{\psi}) \right] \di \xx -
    \int_{\Omega} p^{\varepsilon} \left( \Div_{\xx} \vec{\psi} +
      \frac{1}{\varepsilon} \Div_{\yy} \vec{\psi} \right) \di
    \xx \\
    =  -
    \int_{\Omega} \bT (\varphi^{\varepsilon}):\left[
      \DD_{\xx}(\vec{\psi}) +
      \frac{1}{\varepsilon}\DD_{\yy}(\vec{\psi}) \right] \di \xx.
  \end{split}
\end{align*}
Multiplying by $\varepsilon$ and passing to the limit as $\varepsilon \to
0$ yields: 
\begin{align}
  \label{eq:p442}
  \begin{split}
  &\frac{2}{ \abs{Y}} \int_{\Omega} \int_Y \left[ \DD_{\xx}(\uu^0) +
  \DD_{\yy}(\uu^1) \right] : \DD_{\yy}(\vec{\psi})\di \yy \di
  \xx \\
  &\quad- \frac{1}{\abs{Y}} \int_{\Omega}\int_Y p^0 \Div_{\yy}
  \vec{\psi}\di \yy \di \xx 
  = -\lim_{\varepsilon \to 0}\int_{\Omega} \bT
  (\varphi^{\varepsilon}):\DD_{\yy}(\vec{\psi}) \di \xx.
  \end{split}
\end{align}

Observe that: 
\begin{align*}
&\abs{\int_{\Omega}
    \bT(\varphi^{\varepsilon}(\xx)) : \DD_{\yy}\left( \vec{\psi}
      \left( \xx, \frac{\xx}{\varepsilon} \right) \right) \di \xx - \frac{1}{\abs{Y}} \int_{\Omega} \int_Y \bT^{0}(\xx,\yy) :
  \DD_{\yy} \left( \vec{\psi}(\xx,\yy) \right) \di \yy \di \xx}\\
  &\le \abs{\int_{\Omega}
    \bT(\varphi^{\varepsilon}(\xx)) : \DD_{\yy}\left( \vec{\psi}
      \left( \xx, \frac{\xx}{\varepsilon} \right) \right) \di \xx - \int_{\Omega} \bT^{0} \left( \xx,
      \frac{\xx}{\varepsilon} \right) : \DD_{\yy} \left( \vec{\psi}
  \left( \xx, \frac{\xx}{\varepsilon} \right) \right) \di \xx}\\
  &+ \abs{\int_{\Omega} \bT^{0} \left( \xx,
      \frac{\xx}{\varepsilon} \right) : \DD_{\yy} \left( \vec{\psi}
      \left( \xx, \frac{\xx}{\varepsilon} \right) \right) \di \xx - \frac{1}{\abs{Y}} \int_{\Omega} \int_Y \bT^{0}(\xx,\yy) :
    \DD_{\yy} \left( \vec{\psi}(\xx,\yy) \right) \di \yy \di \xx},
\end{align*}
hence, by \eqref{eq:p466} and \cref{ave-lem}, we deduce
$ 
    \lim_{\varepsilon \to 0} \int_{\Omega}
    \bT(\varphi^{\varepsilon}(\xx)) : \DD_{\yy}\left( \vec{\psi}
      \left( \xx, \frac{\xx}{\varepsilon} \right) \right) \di \xx
    = \frac{1}{\abs{Y}} \int_{\Omega} \int_Y \bT^{0}(\xx,\yy) :
    \DD_{\yy} \left( \vec{\psi}(\xx,\yy) \right) \di \yy \di \xx.
$ 
Therefore, \eqref{eq:p442} becomes: 
\begin{align}
\label{eq:p471}
\int_{\Omega} \int_Y \left( \vec{\sigma}^0 + \bT^0 \right) : \DD_{\yy}(\vec{\psi}) \di \yy \di \xx = 0.
\end{align}

In \eqref{eq:p471}, we let $\vec{\psi}(\xx,\yy) = \phi(\xx)\vec{\theta} \left( \yy \right)$, where $ \phi \in C_c^{\infty}(\Omega)$ and $\vec{\theta} \in C^{\infty}_{c} (Y_f)^d$. Then $\int_{\Omega} \phi (\xx) \left[ \int_Y \left( \vec{\sigma}^0 + \bT^0 \right) : \DD_{\yy}(\vec{\theta}) \di \yy \right] \di \xx = 0$, so $\int_Y \left( \vec{\sigma}^0 + \bT^0 \right) : \DD_{\yy}(\vec{\theta}) \di \yy = 0$ for all $\xx \in \Omega$. Integrating by parts with respect to $\yy$ to obtain: 
\begin{align}
\label{eq:divy}
\Div_{\yy} \left(\vec{\sigma}^0+\bT^0 \right) = 0 \text{ in } \Omega \times Y_f.
\end{align} 

\item Let $\vec{\theta} \in C^{\infty}_{c} (Y)$ such that $\DD_{\yy} (\vec{\theta}) = 0$ in $Y_s$. The same argument as in point \ref{itemDivy} for test function $\vv(\xx) = \vec{\theta} \left(\frac{\xx}{\varepsilon}\right)$ leads to $\int_{\Omega} \int_Y \left( \vec{\sigma}^0 + \bT^0 \right) : \DD_{\yy}(\vec{\theta}) \di \yy \di \xx = 0$.
Integrating by parts with respect to $\yy$ and using \eqref{eq:divy}, we obtain
two balance equations $\int_{\Omega}\int_{\Gamma} \left(\vec{\sigma}^0 + \bT^0 \right) \nn_{\Gamma}\di \Gamma_{\yy} \di \xx = 0$ and $\int_{\Omega} \int_{\Gamma} \left[ \left(
  \vec{\sigma}^0 +
  \bT^0 \right)\nn_{\Gamma} \right] \times \nn_{\Gamma}\di \Gamma_{\yy} \di \xx
  = 0$. 
\medskip

\item
  By \eqref{eq:p439}, we have that 
$\vec{\sigma}^{\varepsilon} = 2 \DD(\uu^{\varepsilon})
-p^{\varepsilon}\II$ two-scale converges to $\vec{\sigma}^0\coloneqq 2 \left( \DD_{\xx}(\uu^0) +
  \DD_{\yy}(\uu^1) \right) - p^0 \II$; and, therefore, it converges to
$ \frac{1}{\abs{Y}} \int_{\Omega} \int_Y \vec{\sigma}^0(\xx,\yy) \di
\yy$ weakly in $L^2(\Omega)$.
By \eqref{eq:p411}, we have $-\Div \vec{\sigma}^{\varepsilon} =
 \Div \bT(\varphi^{\varepsilon})$
in $\Omega_f^{\varepsilon}$.
Moreover, since both $\DD(\uu^{\varepsilon})$ and $p^{\varepsilon} \II$ vanish on $\Omega_s^{\varepsilon}$, we deduce that
$\vec{\sigma}^{\varepsilon} = 0$ in $\Omega_s^{\varepsilon}$.
Therefore:
\begin{align}
\label{eq:p368}
-\Div \vec{\sigma}^{\varepsilon} =
  \charac_{\Omega_f^{\varepsilon}} \Div \bT(\varphi^{\varepsilon}),
  \text{ in } \Omega.
\end{align}
Thus, for any $\vec{\psi} \in H_0^1(\Omega)^d$, using integration by parts over $\Omega$ and \eqref{eq:p413}, we obtain:
\begin{align*}
  \int_{\Omega} \vec{\sigma}^{\varepsilon} : \DD(\vec{\psi}) \di \xx
  = \int_{\Omega} \left( -\Div \vec{\sigma}^{\varepsilon} \right)
  \vec{\psi}\di \xx 
  = - \int_{\Omega_f^{\varepsilon}} 
     \bT(\varphi^{\varepsilon}) : \DD (\vec{\psi})
     \di \xx.
\end{align*}
Taking $\varepsilon \to 0$ and using
    \eqref{eq:p466} together with \cref{ave-lem}
    yield: 
  \begin{align*}
  \begin{split}
    \frac{1}{\abs{Y}} \int_{\Omega} \int_Y \vec{\sigma}^0 :
    \DD(\vec{\psi}) \di \yy \di \xx
    &= 
    - \frac{1}{\abs{Y}}\int_{\Omega}\int_Y
    \bT^{0}:\DD(\vec{\psi})\di \yy \di \xx.
  \end{split}
\end{align*}
Integrating by parts over $\Omega$, we obtain: 
\begin{align*}
\int_{\Omega} \Div \left( \frac{1}{\abs{Y}} \int_Y
  \left( \vec{\sigma}^0 + \bT^{0} \right)\di
  \yy \right) \cdot \vec{\psi} \di \xx
  = 0, \text{ for all }
  \vec{\psi} \in H_0^1(\Omega)^d,
\end{align*}
which implies that 
$ \Div \left( \frac{1}{\abs{Y}} \int_Y \left( \vec{\sigma}^0 + \bT^0
  \right) \di \yy
  \right) = 0, \text{ on } \Omega.$ \medskip

\item We now prove that the triple $(\uu^{\varepsilon},
  p^{\varepsilon}, \varphi^{\varepsilon})$ two-scale converges to
  $(\uu^0, p^0, \varphi^0)$. To this end, we only need to show that the
  limits are unique. The uniqueness of $\varphi^0$ and $\varphi^1$ was
  proven above in \cref{sec:magn-equat-2}. We prove the uniqueness of $\uu^0$ and $p^0$ in this step and in the next step, respectively.

  Let $\vec{\phi}^0 \in
\mathcal{D}(\Omega)$ and $\vec{\phi}^1 \in
\mathcal{D} \left( \Omega, C_{\per}^{\infty}(Y)/\RR
\right)$.  Multiplying \eqref{eq:p368} by $\vec{\phi}(\xx) \coloneqq
\vec{\phi}^0(\xx) + \varepsilon \vec{\phi}^1 \left( \xx,
  \frac{\xx}{\varepsilon} \right)$, and integrating by parts over
$\Omega$, we obtain: 
\begin{align*}
  \int_{\Omega} \vec{\sigma}^{\varepsilon}: \DD(\vec{\phi}) \di \xx
  &=
-\int_{\Omega} \left( \Div \vec{\sigma}^{\varepsilon} 
  \right) \cdot \vec{\phi}\di \xx
  =  -
    \int_{\Omega_f^{\varepsilon}} \bT (\varphi^{\varepsilon}):\DD(\vec{\phi}) \di \xx.
\end{align*}
Note that $\DD(\vec{\phi}) = \DD(\vec{\phi}^0)(\xx) +
\varepsilon \DD_{\xx}(\vec{\phi}^1)+ \DD_{\yy}(\vec{\phi}^1)\left(
  \xx,\frac{\xx}{\varepsilon} \right)$.  Taking $\varepsilon\to0$, using $\vec{\sigma}^{\varepsilon}
\tscale \vec{\sigma}^0$ and \eqref{eq:p466}, we obtain:
\begin{align}
\label{eq:p380}
  \begin{split}
\frac{1}{\abs{Y}} &\int_{\Omega} \int_Y \vec{\sigma}^0 : \left[
  \DD(\vec{\phi}^0)+\DD_{\yy}(\vec{\phi}^1)\right] \di \yy \di
  \xx 
  \\
  &=  -
\frac{1}{\abs{Y}} \int_{\Omega}\int_Y \bT^0:\left[\DD(\vec{\phi}^0) +
  \DD_{\yy}(\vec{\phi}^1)  \right] \di \yy \di \xx,
  \end{split}
\end{align}
Suppose further that $\Div \vec{\phi}^0 = 0,$ $\Div_{\yy} \vec{\phi}^1 = 0$, then:
\begin{align}
\label{eq:p374}
  \begin{split}
\frac{2}{\abs{Y}} &\int_{\Omega} \int_Y \left[
  \DD(\uu^0)+\DD_{\yy}(\uu^1)\right] : \left[
  \DD(\vec{\phi}^0)+\DD_{\yy}(\vec{\phi}^1)\right] \di \yy \di
  \xx \\
&=  -
   \frac{1}{\abs{Y}} \int_{\Omega}\int_Y \bT^0:\left[\DD(\vec{\phi}^0) +
  \DD_{\yy}(\vec{\phi}^1)  \right] \di \yy \di \xx.
  \end{split}
\end{align}

Consider the space:
\begin{align}
\label{defH}
\calH 
&\coloneqq 
\Set*{(\vv^0,\vv^1) \in H_{0}^{1}(\Omega)^{d} \times L^2 (\Omega, H_{\per}^{1}(Y) /\RR) \given
\begin{aligned}
\Div \vv^0 &= 0 \text{ in } \Omega\\
\Div_{\yy} \vv^1 &= 0 \text{ in } \Omega \times Y\\
\DD(\vv^0) +\DD_{\yy}(\vv^1) &=0 \text{ in } \Omega \times Y_s
\end{aligned}
}.
\end{align}
Let $\calH$ be endowed with the inner product: 
\begin{align}
\label{innerpH}
\innerp{
(\vv^0,\vv^1), (\ww^0,\ww^1)
}_{\calH}
\coloneqq
\int_{\Omega} \nabla \vv^0 : \nabla \ww^0 \di \xx
+ \frac{1}{\abs{Y}}\int_{\Omega} \int_{Y} \nabla_{\yy} \vv^1 : \nabla_{\yy} \ww^1 \di \yy \di \xx,  
\end{align}
for all $(\vv^0,\vv^1),~ (\ww^0,\ww^1)$ in $\calH$.
Then it can be shown that $\calH$ is a Hilbert space.
By density, \eqref{eq:p374} holds for all $(\vec{\phi}^0,
\vec{\phi}^1)$ in $\calH$.
Now, let: 
\begin{align*}
    b((\vv^0,\vv^1),(\ww^0,\ww^1))
    \coloneqq
    \frac{2}{\abs{Y}} \int_{\Omega} \int_Y \left[
      \DD(\vv^0)+\DD_{\yy}(\vv^1)\right] : \left[
      \DD(\vec{\ww}^0)+\DD_{\yy}(\vec{\ww}^1)\right] \di \yy \di
    \xx,
\end{align*}
for $(\vv^0,\vv^1)$ and $(\ww^0,\ww^1)$ in $\calH$. Clearly, the
bilinear form $b$ and the linear map
$(\uu^0,\uu^1) \mapsto  -
   \frac{1}{\abs{Y}} \int_{\Omega} \int_Y \bT^0:\left[\DD(\uu^0) +
  \DD_{\yy}(\uu^1)  \right] \di \yy \di \xx$ are continuous by
H\"{o}lder's inequality. For $(\uu^0, \uu^1)$ in $\calH$, we write
$\uu^0 = (u^0_1, \ldots,u^0_d)$, $\uu^1 = (u^1_1, \ldots, u^1_d)$, and
let $\nn_{\Gamma} = (n_1,\ldots,n_d)$ be the unit normal vector on
$\Gamma$, then:
\begin{align*}
\int_{\Omega} \int_Y \DD(\uu^0) : \DD_{\yy}(\uu^1)\di \yy \di
  \xx
&= \int_{\Omega} \int_Y \frac{\partial}{\partial y_j} \left( u^1_i
        \frac{\partial u^0_i}{\partial x_j} \right)\di \yy \di \xx
= \int_{\Omega} \int_{\partial Y} u^1_i n_j \frac{\partial
   u^0_i}{\partial x_j} \di s_{\yy} \\
&= 0,
\end{align*}
since $\uu^1$ is periodic with respect to $\yy$. Therefore, we infer that: 
\begin{align*}
  \begin{split}
    b((\uu^0,\uu^1),(\uu^0,\uu^1)) 
    &= \frac{2}{ \abs{Y}} \left[
      \norm{\DD(\uu^0)}^2_{L^2(\Omega \times Y)} +
      \norm{\DD_{\yy}(\uu^1)}^2_{L^2(\Omega \times
        Y)}+\phantom{\int_{\Omega} s}\right.\\
    &\phantom{=} \left. \mathrel{+} 2 \int_{\Omega} \int_Y \DD(\uu^0) :
      \DD_{\yy}(\uu^1)\di \yy \di \xx \right]
    \ge C \norm{(\uu^0,\uu^1)}^2_{\calH},
  \end{split}
\end{align*}
and hence, $b$ is coercive.  The Lax-Milgram Theorem is then applied to obtain the existence and uniqueness of $(\uu^0,\uu^1) \in \calH$ - the
solution of \eqref{eq:p374} - for any $(\vec{\phi}^0,\vec{\phi}^1) \in
\calH$. This implies that the full sequence $\uu^\varepsilon$, not just up to a subsequence, two-scale converges to $\uu^0$. 

\item We now show that $p^0$ is unique.
Define: 
\begin{align}
\label{defKL}
\begin{split}
\calK
&\coloneqq 
\left\{(\vv^0,\vv^1) \in H_{0}^{1}(\Omega)^{d} \times L^2 \left(\Omega, H_{\per}^{1}(Y)^d /\RR\right) \colon \right.\\ &\qquad \qquad \qquad\left.
\DD(\vv^0) +\DD_{\yy}(\vv^1) =0 \text{ in } \Omega \times Y_s
\right\},\\
\calL &\coloneqq 
    \left\{
    (q^0,q^1) \in L^2_0(\Omega) \times L^2\left(\Omega, L^2_{\per}(Y)/\RR\right) \colon \right.\\
    &\qquad \qquad \qquad\left.\exists (\vv^0,\vv^1) \in\calK \text{ s.t. } q^0 = \Div \vv^0,~q^1 = \Div_{\yy} \vv^1
    \right\},
\end{split}
\end{align}
and $\innerp{
(q^0,q^1), (p^0,p^1)
}_{\calL}
\coloneqq
\int_{\Omega} q^0(\xx) p^0(\xx) \di \xx
+ \frac{1}{\abs{Y}}\int_{\Omega} \int_{Y} q^1(\xx,\yy) p^1(\xx,\yy) \di \yy \di \xx. $
Then it is clear that $\left( \calK, \innerp{\cdot,\cdot}_{\calH} \right)$ and $\left( \calL, \innerp{\cdot,\cdot}_{\calL}\right)$ are Hilbert spaces. 
For $(\vv^0,\vv^1) \in \calK,$ define $J_b \colon \calK \to \RR$ and $G_b \colon \calK \to \calL$ as follows:
\begin{align*}
J_b(\vv^0,\vv^1) 
&\coloneqq 
b((\vv^0,\vv^1),(\vv^0,\vv^1)) +
   \frac{1}{\abs{Y}} \int_{\Omega} \int_Y \bT^0:\left[\DD(\uu^0) +
  \DD_{\yy}(\uu^1)  \right] \di \yy \di \xx,\\
G_b(\vv^0,\vv^1) &\coloneqq \left(\Div \vv^0, \Div_{\yy} \vv^1\right).
\end{align*}
Since $b$ is symmetric, by a similar argument as in the proof of \cref{sec:coupled-equation-4}, we have: $$(\uu^0,\uu^1) = \argmin_{(\vv^0,\vv^1) \in \calK} J_b(\vv^0,\vv^1),\qquad \text{ subject to } G_b(\vv^0,\vv^1) = 0.$$ 
Clearly, the following holds:
$
\left\langle G'_b(\vv^0,\vv^1), (\ww^0,\ww^1) \right\rangle= \left(\Div \ww^0, \Div_{\yy} \ww^1 \right),
$
where the left hand side is the pairing between the Fr\'{e}chet derivative $G_b'(\vv^0,\vv^1)$ acting on $(\ww^0,\ww^1)$. This yields $G_b \in C^1(\calK, \calL).$
Moreover, $G'_b(\uu^0,\uu^1) \colon \calK \to \calL$ is surjective by construction.
By the Lagrange Multiplier theorem, there exists $(\bar{q},q) \in \calL^* \cong \calL$  such that: 
\begin{align}
\label{lagrange}
    J'_b (\uu^0,\uu^1) + (\bar{q},q) \circ G'_b(\uu^0,\uu^1) = 0,
\end{align}
where $\circ$ denotes the composition of functions.
Testing \eqref{lagrange} against $(\vec{\phi}^0,\vec{\phi}^1) \in \calK$:
\begin{align*}
    &\frac{1}{\abs{Y}} \int_{\Omega} \int_Y \vec{\sigma}^0 : \left[
  \DD(\vec{\phi}^0)+\DD_{\yy}(\vec{\phi}^1)\right] \di \yy \di
  \xx 
  \\
  & \qquad
  +
\frac{1}{\abs{Y}} \int_{\Omega}\int_Y \bT^0:\left[\DD(\vec{\phi}^0) +
  \DD_{\yy}(\vec{\phi}^1)  \right] \di \yy \di \xx \\
  &\qquad{}+{}
  \frac{1}{\abs{Y}} \int_{\Omega} \int_{Y_f} \bar{q} 
  \Div \vec{\phi}^0\di \yy \di
  \xx
  + \frac{1}{\abs{Y}} \int_{\Omega} \int_{Y_f} q 
  \Div_{\yy} \vec{\phi}^1\di \yy \di
  \xx = 0
\end{align*}
Let $\vec{\phi}^0 = 0$ and $\vec{\phi}^1 \in
\mathcal{D} \left( \Omega, C_{\per}^{\infty}(Y_f)^d/\RR
\right)$, extended by zero into $Y_s$, then: 
\begin{align*}
   \frac{1}{\abs{Y}} \int_{\Omega} \int_Y \left[
  \DD(\uu^0)+\DD_{\yy}(\uu^1)\right]  : \DD_{\yy}(\vec{\phi}^1) \di \yy \di
  \xx 
  + \frac{1}{\abs{Y}} \int_{\Omega}\int_Y \bT^0:
  \DD_{\yy}(\vec{\phi}^1)  \di \yy \di \xx\\
  + \frac{1}{\abs{Y}} \int_{\Omega} \int_{Y_f} q 
  \Div_{\yy} \vec{\phi}^1\di \yy \di
  \xx = 0.
\end{align*}
On the other hand, let $\vec{\phi}^0 = 0$ in \eqref{eq:p380}, then for all $\vec{\phi}^1 \in
\mathcal{D} \left( \Omega, C_{\per}^{\infty}(Y)^d/\RR
\right)$:
\begin{align*}
  \begin{split}
\frac{1}{\abs{Y}} \int_{\Omega} \int_Y p^0\II : 
  \DD_{\yy}(\vec{\phi}^1)\di \yy \di
  \xx 
  +\frac{1}{\abs{Y}} \int_{\Omega} \int_Y \left[
  \DD(\uu^0)+\DD_{\yy}(\uu^1)\right]  : \DD_{\yy}(\vec{\phi}^1) \di \yy \di
  \xx
  \\
  =  -
\frac{1}{\abs{Y}} \int_{\Omega}\int_Y \bT^0:
  \DD_{\yy}(\vec{\phi}^1)  \di \yy \di \xx,
  \end{split}
\end{align*}
Therefore, for all $\vec{\phi}^1 \in
\mathcal{D} \left( \Omega, C_{\per}^{\infty}(Y_f)^d/\RR
\right)$, we have
$ 
    \frac{1}{\abs{Y}} \int_{\Omega} \int_{Y_f} q 
  \Div_{\yy} \vec{\phi}^1\di \yy \di
  \xx =
  \frac{1}{\abs{Y}} \int_{\Omega} \int_{Y_f} p^0\II : 
  \DD_{\yy}(\vec{\phi}^1)\di \yy \di
  \xx,
$ 
which yields $q = p^0$ in $\Omega \times Y_f$. Thus $q = p^0$ in $\Omega \times Y$ because both functions are zero inside $Y_s$. 
This implies that the full sequence $p^\varepsilon$, not just up to a subsequence, two-scale converges to $p^0$.
\end{enumerate}
\end{proof}

\begin{remark}
\label{sec:coupled-equation-1}
If the coupling parameter $S \equiv 0$, the proof presented above does, in particular, justify the formal asymptotic results of
\cite{levySuspensionSolidParticles1983} for the case of moving rigid particles in a steady viscous flow.
\end{remark}

\subsection{Cell Problems and Corrector Results}
\label{sec:cell-probl-corr}

We recall the cell problem for the magneto-static equations is presented in 
\eqref{eq:p453}.
These equations, together with \eqref{eq:p452}, \eqref{eq:p455} and
\eqref{eq:p456}, uniquely determine  $\varphi^0$ and $\varphi^1$, thus they also determine $\bT^0$, by \eqref{eq:502}. Moreover, one can write:
\begin{align}
  \label{eq:512}
  \begin{split}
  \bT^0 (\xx,\yy)
  = S \frac{\partial \varphi^0}{\partial x_i}
    \frac{\partial\varphi^0}{\partial x_j} \mu(\yy) \left(
    \left[ \ee^i + \nabla_{\yy} \omega^i \right] \otimes \left[ \ee^j
    + \nabla_{\yy} \omega^j \right] \vphantom{\frac{1}{2}}\right.\\
    \left.
    - \frac{1}{2} \left[ \ee^i +
    \nabla_{\yy} \omega^i \right] \cdot \left[ \ee^j + \nabla_{\yy}
    \omega^j \right] \II  
    \right).
  \end{split}
\end{align}

Equations 
\eqref{eq:p476} and \eqref{eq:496} suggest that it is possible to write $\uu^1$ as a
function of $\uu^0$ and $\varphi^0$.  To achieve this, let us introduce, for $1 \le i,j \le d$, the matrix $\bQ^{ij}$ satisfying: 
\begin{itemize}
\item If $i = j$, then $\bQ^{ii}_{ii} = 1$ and the rest of the entries are zero.
\item If $i \ne j$, then $\bQ^{ij}_{ij} = \bQ^{ij}_{ji} = \frac{1}{2}$
  and the rest of the entries are zero.
\end{itemize}
In short, the above assumptions imply that $\bQ^{ij}_{mn} = \frac{1}{2} \left(
  \delta_{im}\delta_{jn} + \delta_{in}\delta_{jm} \right)$. Now, consider the vector $\bP^{ij}$ defined by $\bP^{ij}_k \coloneqq y_j
\delta_{ik}$.  Observe that $\DD_{\yy} \left( \bP^{ij} \right) = \bQ^{ij}$. Indeed, write $\bP^{ij}=\bP^{ij}_{n} \ee_n = y_j  \delta_{in} \ee_n,$ then $\nabla \bP^{ij} = \frac{\partial}{\partial y_m} (y_j \delta_{in} ) \ee_n \otimes \ee_m = \delta_{jm} \delta_{in} \ee_n \otimes \ee_m$. Therefore, $\left(\nabla \bP^{ij} \right)^{\top} = \delta_{jn} \delta_{im} \ee_n \otimes \ee_m$, so $\DD_{\yy} (\bP^{ij}) = \frac{1}{2} (\delta_{jm} \delta_{in} + \delta_{jn} \delta_{im}) \ee_n \otimes \ee_m.$

Now, we write $\uu^1$ as a function of $\uu^0$ and $\varphi^0$, by setting: 
\begin{align}
\label{eq:497}
  \uu^1(\xx, \yy)
  = -\DD\left(\uu^0(\xx)\right)_{ij} \vec{\chi}^{ij}(\yy) +
  S\frac{\partial \varphi^0}{\partial x_i}(\xx) \frac{\partial
  \varphi^0}{\partial x_j}(\xx) \vec{\xi}^{ij}(\yy),
\end{align}
where $\vec{\chi}^{ij}, \vec{\xi}^{ij} \in H^1_{\per}(Y)^d/\RR$. It follows that: 
\begin{align}
  \label{eq:498}
  \begin{split}
  \DD \left( \uu^0 \right) + \DD_{\yy} \left( \uu^1 \right)
  &= \DD \left( \uu^0 \right)_{ij} \DD_{\yy}\left( \bP^{ij} -
    \vec{\chi}^{ij} \right)
  +  S\frac{\partial \varphi^0}{\partial x_i} \frac{\partial
  \varphi^0}{\partial x_j} \DD_{\yy}\left( \vec{\xi}^{ij} \right).
  \end{split}
\end{align}
Now, let
$
  p^0(\xx,\yy)
  = 2\DD \left( \uu^0(\xx) \right)_{ij} q^{ij}(\yy)
  -  S\frac{\partial \varphi^0}{\partial x_i}(\xx) \frac{\partial
  \varphi^0}{\partial x_j}(\xx) r^{ij}(\yy) + \pi^0(\xx),
  $
where $\pi^0 \in L^2_0(\Omega)$ and $q^{ij}, r^{ij} \in L^2_{\per}(Y)/\RR$. 
The formula, together with
\eqref{eq:498}, yield:
\begin{align*}
  \begin{split}
    \vec{\sigma}^0 &= 2\DD \left( \uu^0\right)_{ij} \left[
       \DD_{\yy} \left( \bP^{ij} - \vec{\chi}^{ij}
      \right) - q^{ij}\II
    \right]
    - \pi^0\II
    + S\frac{\partial \varphi^0}{\partial x_i} \frac{\partial
      \varphi^0}{\partial x_j} \left( \DD_{\yy} \left( \vec{\xi}^{ij}
      \right) + r^{ij}\II \right).
  \end{split}
\end{align*}
Let's define the following matrices: 
\begin{align}
\label{eq:513}
  \calA^{ij}
  &\coloneqq
    \frac{1}{\abs{Y}} \int_Y \DD_{\yy} \left(
    \bP^{ij}-\vec{\chi}^{ij} \right)\di \yy 
    ~\text{ and }~
  \mathcal{B}^{ij}
  \coloneqq
    \frac{1}{\abs{Y}} \int_Y \left( \DD_{\yy}(\vec{\xi}^{ij}) 
    +  \vec{\tau}^{ij}\right)\di \yy,
\end{align}
where 
    $\vec{\tau}^{ij}
    \coloneqq
    \mu (\yy) \left[ (\ee^i + \nabla_{\yy} \omega^i)  \otimes ( \ee^j
    + \nabla_{\yy} \omega^j) - \frac{1}{2} ( \ee^i +
    \nabla_{\yy} \omega^i ) \cdot ( \ee^j + \nabla_{\yy}
    \omega^j ) \II \right]$. 
    
Then, from \eqref{eq:p476}, \eqref{eq:494}, {and  $\Div \left(\int_Y q^{ij}(\yy) \II \di \yy \right) = 0= \Div \left(\int_Y r^{ij}(\yy) \II \di \yy \right),$} we conclude that: In $\Omega$,
\begin{align}
  \label{eq:501}
\Div \left[ 
  2\calA^{ij}\DD \left( \uu^0 \right)_{ij} - \pi^0 \II+ 
  S\mathcal{B}^{ij} \frac{\partial \varphi^0}{\partial x_i}
  \frac{\partial \varphi^0}{\partial x_j}\right]
  = 0
  \quad\text{ and }\quad
  \Div \uu^0 
  = 0.
\end{align}
As it follows from \eqref{eq:501},  $\calA \coloneqq \left\{ \calA^{ij}_{mn}\right\}_{1 \le i,j,m,n \le d}$ is the \emph{effective viscosity}, and it is a fourth rank tensor. 

From \eqref{eq:p477}, \eqref{eq:496}, \eqref{bdr-cell2}, \eqref{eq:a494} and \eqref{bdr-cell3}, the cell problems for $\vec{\chi}^{ij}, \vec{\xi}^{ij} \in H^1_{\per}(Y)^d/\RR$ and
$q^{ij}, r^{ij} \in L^2(Y)/\RR$ are given by: 
\begin{align}
\label{eq:504}
  \begin{split}
    &\Div_{\yy} \left[ \DD_{\yy} \left(\bP^{ij} - \vec{\chi}^{ij}
      \right) + q^{ij}\II \right]
    = 0 \text{ in } Y_f,\\
    &\Div_{\yy} \vec{\chi}^{ij}
    = 0 \text{ in }Y, \qquad
    \DD_{\yy} \left( \bP^{ij} - \vec{\chi}^{ij} \right)
    = 0 \text{ in } Y_s,\\
    &\int_{\Gamma}  \left[
      \DD_{\yy} \left( \bP^{ij} - \vec{\chi}^{ij}
      \right) - q^{ij}\II
    \right] \nn_{\Gamma} \di \Gamma_{\yy} 
    =0=
    \int_{\Gamma}  \left[
       \DD_{\yy} \left( \bP^{ij} - \vec{\chi}^{ij}
      \right) - q^{ij}\II
    \right] \nn_{\Gamma} \times \nn_{\Gamma} \di \Gamma_{\yy},
    \end{split}
    \end{align}
\begin{align}
\label{eq:504b}
  \begin{split}
    &\Div_{\yy} \left[  \DD_{\yy} \left( \vec{\xi}^{ij}
      \right) + r^{ij}\II +  \vec{\tau}^{ij} \right]
    = 0 \text{ in } Y_f,\\
    &\Div_{\yy} \vec{\xi}^{ij}
    = 0 \text{ in }Y, \qquad
    \DD_{\yy} \left( \vec{\xi}^{ij} \right)
    = 0 \text{ in } Y_s,\\
    &\int_{\Gamma}  \left[
      \DD_{\yy} \left( \vec{\xi}^{ij}
      \right) + r^{ij}\II +  \vec{\tau}^{ij}
    \right] \nn_{\Gamma} \di \Gamma_{\yy}  
    =0=
    \int_{\Gamma}  \left[
      \DD_{\yy} \left( \vec{\xi}^{ij}
      \right) + r^{ij}\II +  \vec{\tau}^{ij}
    \right] \nn_{\Gamma} \times \nn_{\Gamma} \di \Gamma_{\yy}.
  \end{split}
\end{align}

\begin{lemma}
\label{regularity-u0}
The coefficients of the effective tensor $\calA$ can be written as:
\begin{align}
\label{eq:a515}
  \calA^{ij}_{mn}
  \coloneqq \frac{1}{\abs{Y}} \int_Y 
  \DD_{\yy}(\bP^{ij}-\vec{\chi}^{ij}) : \DD_{\yy}(\bP^{mn}-\vec{\chi}^{mn})
  \di \yy.
\end{align}
Therefore, $\calA$ is symmetric, i.e. $\calA^{ij}_{mn} = \calA^{mn}_{ij} = \calA^{ji}_{mn} = \calA^{ij}_{nm}$.
Moreover, $\calA$ satisfies the Legendre-Hadamard condition (or strong ellipticity condition), i.e. there exist $\lambda > 0$ such that, for all $\vec{\zeta}, \vec{\eta} \in \RR^d$, one has
$ 
    \calA^{ij}_{mn} \zeta_i \zeta_m \eta_j \eta_n \ge \lambda \abs{\vec{\zeta}}^2 \abs{\vec{\eta}}^2.
$ 
As a consequence, the system \eqref{eq:501} with homogeneous Dirichlet boundary condition has a unique solution $(\uu^0,\pi^0) \in C^1(\bar{\Omega})^d \times C(\bar{\Omega})$ (here $\pi^0$ is defined up to a constant).
\end{lemma}

\begin{proof}
We follow the same line of argument as in \cref{regularity-phi0}. 
Testing the first equation of \eqref{eq:504} with $\vec{\chi}^{mn}$, and using the incompressibility condition to eliminate the pressure term, we obtain
$ 
    \int_Y \DD_{\yy} (\bP^{ij}-\vec{\chi}^{ij}) : \DD_{\yy}(\vec{\chi}^{mn}) \di \yy = 0.
$ 
Since $\DD_{\yy}(\bP^{ij}) = \bQ^{ij} = \bQ^{ji} =  \DD_{\yy}(\bP^{ji})$, \eqref{eq:504} implies that $\vec{\chi^{ij}} = \vec{\chi^{ji}}$. By \eqref{eq:513}, we have $\calA^{ij}_{mn} = \calA^{ji}_{mn}$. 
Taking into account with the first equation of \eqref{eq:513}, we obtain \eqref{eq:a515}.
This representation \eqref{eq:a515} and the fact that $\calA^{ij}_{mn} = \calA^{ji}_{mn}$ imply that $\calA$ is symmetric.

Now, for all $\vec{\zeta}, \vec{\eta} \in \RR^d$:
\begin{align*}
\Upsilon (\vec{\zeta}, \vec{\eta}) 
&\coloneqq
 \calA^{ij}_{mn} \zeta_i \zeta_m \eta_j \eta_n 
  = \frac{1}{\abs{Y}} \int_Y 
  \DD_{\yy}(\bP^{ij}-\vec{\chi}^{ij})\zeta_i \eta_j : \DD_{\yy}(\bP^{mn}-\vec{\chi}^{mn})\zeta_m \eta_n 
  \di \yy\\
  &= \frac{1}{\abs{Y}} \int_Y 
  \abs{\DD_{\yy}(\bP^{ij}-\vec{\chi}^{ij})\zeta_i \eta_j }^2 
  \di \yy \ge 0.
\end{align*}
Therefore, $\Upsilon \vert_{\partial B(0,1)^2} (\vec{\zeta},\vec{\eta}) = 0$ if and only if $\DD_{\yy}(\bP^{ij}-\vec{\chi}^{ij}) = 0$ a.e. in $Y$. The latter implies that $\vec{\chi}^{ij} = \bP^{ij} {}+{} \text{constant}$, which is a contradiction, since it wouldn't be a periodic function. Let $\lambda$ be the minimum of the continuous function $\Upsilon$ on the compact set $\partial B(0,1)^2$, then the argument above yields $\lambda > 0$. By scaling, we have $\calA$ is strongly elliptic.
For the last claim, since $\Omega$ is of class $C^3$, $\bg \in H({\Omega})^d$ and $\varphi^0 \in W^{2,4}({\Omega})$,
the solution $(\uu^0, \pi^0)$ of \eqref{eq:501} belongs to $H^3(\Omega)^d \times H^2(\Omega) \subset C^1(\bar{\Omega})^d \times C(\bar{\Omega})$ by \cite[Theorem 4.1]{huyExistenceRegularitySolutions2006}. 
\end{proof}

\begin{lemma}[First order corrector result for the velocity]
\label{sec:cell-probl-corr-2}
Let $\uu^{\varepsilon}, \varphi^{\varepsilon}, p^{\varepsilon}, \uu^0,$ $\varphi^0, p^0, $ $\uu^1, \varphi^1$ be as in \cref{thm:main}. Then: 
\begin{align}
\label{eq:522}
\lim_{\varepsilon \to 0} \norm{\DD(\uu^{\varepsilon})(\cdot) - \DD
  (\uu^0)(\cdot) - \DD_{\yy} (\uu^1) \left( \cdot,
  \frac{\cdot}{\varepsilon} \right)}_{L^2(\Omega)}
  = 0.
\end{align}
\end{lemma}

\begin{proof}
\begin{enumerate}[wide]
\item We write
  \begin{align}
    \label{eq:548}
    2 \norm{\DD(\uu^{\varepsilon}) - \DD(\uu^0) -
    \DD_{\yy}(\uu^1)}^2_{L^2(\Omega)}
    = \mathcal{J}_1 - 2\mathcal{J}_2 + \mathcal{J}_3, \,\,\text{where:}
  \end{align}
\begin{align*}
  \mathcal{J}_1
  &\coloneqq 2 \int_{\Omega} \DD (\uu^{\varepsilon})(\xx) : \DD
    (\uu^{\varepsilon})(\xx) \di \xx,\\
  \mathcal{J}_2
  &\coloneqq 2 \int_{\Omega} \left[ \DD(\uu^0)(\xx) +
    \DD_{\yy}(\uu^1) \left( \xx, \frac{\xx}{\varepsilon}\right)
    \right] : \DD(\uu^{\varepsilon}) (\xx) \di\xx,\\
  \mathcal{J}_3
  &\coloneqq 2\int_{\Omega}  \left[ \DD(\uu^0)(\xx) +
    \DD_{\yy}(\uu^1) \left( \xx, \frac{\xx}{\varepsilon}\right)
    \right] :  \left[ \DD(\uu^0)(\xx) +
    \DD_{\yy}(\uu^1) \left( \xx, \frac{\xx}{\varepsilon}\right)
    \right] \di \xx.
\end{align*}

\item By letting $\vv = \uu^{\varepsilon}$ in \eqref{eq:p436}, we deduce that:
\begin{align}
\label{eq:523}
  \begin{split}
    \lim_{\varepsilon \to 0} \mathcal{J}_1
    &= -
    \lim_{\varepsilon \to 0} \int_{\Omega} \bT (\varphi^{\varepsilon})
    : \DD(\uu^{\varepsilon}) \di \xx.
  \end{split}
\end{align}

Next, we compute the last limit in \eqref{eq:523}. Fix $s>4$ in \cref{sec:coupled-equation-6}.
\begin{itemize}[wide]

\item 
On the one hand, for $1\le i \le d$, by \cite[Corollary
3.5]{zhangInteriorHolderGradient2013a}, we have $\omega^i \in
W^{1,\infty}(Y)$, so $\nabla \omega_i \in
L^{\infty}(Y) \subset L^{s}(Y)$. Let $C = \max_{1\le i \le d}
\norm{\nabla \omega_i}_{L^{s}(Y)},$ then:
\begin{align}
\label{eq:543}
\norm{\nabla \omega^i}_{L^{s}(Y)} \le C \text{ for all } 1 \le
  i \le d.
\end{align}
Let $r = \frac{s}{2} > 2$, then $\bT^0 \left( \cdot, \frac{\cdot}{\varepsilon} \right)$ is
bounded in $L^r(\Omega)$ by \eqref{eq:512} and \cref{eq:543}. Moreover, from \cref{sec:coupled-equation-6} and the definition of $\bT(\varphi^\varepsilon)$ in \eqref{eq:p426}, 
we have
$
  \norm{\bT(\varphi^{\varepsilon})}_{L^{r}(\Omega) }
$ is bounded.
On the other hand, \eqref{eq:p466} and \cite[Theorem
4.9]{brezisFunctionalAnalysisSobolev2011} imply, up to a subsequence, 
$
\lim_{\varepsilon \to 0}  \left(\bT (\varphi^{\varepsilon}(\xx)) -
  \bT^0 \left( \xx,\frac{\xx}{\varepsilon} \right) \right)=0.
$ 
Therefore, by \cite[Exercise 4.16]{brezisFunctionalAnalysisSobolev2011}, we obtain:
\begin{align}
  \label{eq:545}
  \lim_{\varepsilon \to 0} \norm{\bT (\varphi^{\varepsilon}) - \bT^0
  \left( \cdot, \frac{\cdot}{\varepsilon} \right)}_{L^2(\Omega)^{d \times d}} = 0.
\end{align}
Note that this convergence is stronger than \eqref{eq:p466}. This estimate justifies our choice of $s > 4$ in \cref{sec:coupled-equation-6}.

\item By rewriting and taking the limit as $\varepsilon\to0$, we have:
\begin{align}
\label{eq:546}
\begin{split}
    &\int_{\Omega} \bT(\varphi^{\varepsilon})
    : \DD(\uu^{\varepsilon}) \di \xx\\
    &= 
    \int_{\Omega} \left[\bT(\varphi^{\varepsilon})-\bT^0 \left(\xx, \frac{\xx}{\varepsilon}\right) \right]
    : \DD(\uu^{\varepsilon}) \di \xx
    + \int_{\Omega} \bT^0 \left(\xx, \frac{\xx}{\varepsilon}\right) 
    : \DD(\uu^{\varepsilon}) \di \xx\\
    &\xrightarrow[]{\varepsilon \to 0}
    \frac{1}{\abs{Y}} \int_{\Omega} \int_Y \bT^0 (\xx,\yy) : \left[
      \DD(\uu^0) + \DD_{\yy}(\uu^1) \right] \di \yy \di \xx.
\end{split}
\end{align}
Here, the first integral converges to zero due to \eqref{eq:545} and because $\DD(\uu^\varepsilon)$ is bounded; while in the second integral, we regard $\bT^0  \in L^2_{\per}\left(Y, C(\bar{\Omega})\right)$ as a test function for $\DD(\uu^\varepsilon) \tscale \DD(\uu^0) +
  \DD_{\yy}(\uu^1)$.
From \eqref{eq:p476}, \eqref{eq:523} and \eqref{eq:546}, we conclude that: 
\begin{align}
\label{eq:547}
  \lim_{\varepsilon \to 0} \mathcal{J}_1
  = \frac{2}{ \abs{Y}} \int_{\Omega} \int_Y \left[ \DD(\uu^0) +
    \DD_{\yy}(\uu^1) \right] : \left[ \DD(\uu^0) + \DD_{\yy}(\uu^1)
    \right] \di \yy \di \xx.
\end{align}

\end{itemize}

\item 
Substituting \eqref{eq:498} into $\mathcal{J}_2$ and $\mathcal{J}_3$, and using the definition of two-scale convergence for $\mathcal{J}_2$ and \cref{ave-lem} for $\mathcal{J}_3$, we obtain: 
\begin{align}
\label{eq:524}
\begin{split}
  \lim_{\varepsilon \to 0} \mathcal{J}_2 
  &= \lim_{\varepsilon \to 0} \mathcal{J}_3\\
  &=  \frac{2}{ \abs{Y}} \int_{\Omega} \int_Y \left[ \DD(\uu^0) +
    \DD_{\yy}(\uu^1) \right] : \left[ \DD(\uu^0) + \DD_{\yy}(\uu^1)
    \right] \di \yy \di \xx.
\end{split}
\end{align}
Here, we used that $\DD(\uu^0)_{ij} \in C (\bar{\Omega})$ and
$\frac{\partial \varphi^0}{\partial x_i} \in C (\bar{\Omega})$, which
follow from \cref{regularity-u0} and \cref{regularity-phi0}.
Putting \eqref{eq:548}, \eqref{eq:547} and \eqref{eq:524} together, we obtain \eqref{eq:522}.
\end{enumerate}
\end{proof}

Finally, in the following corollary, we synthesize the results of Theorem \ref{thm:main} and the cell problems \eqref{eq:504}-\eqref{eq:504b}.
\begin{corollary} \label{cor:main}
Let $(\varphi^{\varepsilon}, \uu^{\varepsilon},p^{\varepsilon}) \in (H^1(\Omega)/\RR) \times H_0^1(\Omega)^d \times L_0^2(\Omega)$ be the solution of \eqref{eq:p410}. Then: 
    $ 
    \varphi^\varepsilon 
    \rightharpoonup \varphi^0 \text{ in } H^1(\Omega)/\RR,~
    \uu^\varepsilon
    \rightharpoonup \uu^0 \text{ in } H_0^1(\Omega)^d,~
    p^\varepsilon 
    \rightharpoonup \pi^0 \text{ in } L_0^2(\Omega),
    $
where $\varphi^0, \uu^0$ and $\pi^0$ are solutions of: 
\begin{align}
\label{summary-eqn}
\begin{aligned}
    -\Div \left( \mu^{\eff} \nabla\varphi^0 \right) 
    &= {0} &&\text{ in }\Omega,\\
    \left( \mu^{\eff} \nabla\varphi^0 \right) \cdot \nn_{\partial\Omega}
    &= \kk \cdot \nn_{\partial\Omega} &&\text{ on }\partial\Omega,\\
    \Div \left[ 
  2\calA^{ij}\DD \left( \uu^0 \right)_{ij} - {\pi^0} + 
  S\mathcal{B}^{ij} \frac{\partial \varphi^0}{\partial x_i}
  \frac{\partial \varphi^0}{\partial x_j}\right]
  &= 0 &&\text{ in }\Omega,\\
  \Div \uu^0 
  &= 0 &&\text{ in }\Omega,
  \end{aligned}
\end{align}
with $\mu^{\eff}$ defined by \eqref{eq:p456} and $\calA^{ij},$ $\mathcal{B}^{ij}$, $1 \le i,j \le d$ defined in \eqref{eq:513}.
\end{corollary}

\section{Conclusions}
\label{sec:conclusions}
The results obtained above in \cref{thm:main} and \cref{cor:main} demonstrate the \emph{effective} response of a viscous fluid with a locally periodic array of paramagnetic/diamagnetic particles suspended in it, given by the system of equations \eqref{eq:p410}. The effective equations are described by \eqref{summary-eqn} in  \cref{cor:main}, with effective coefficients given by \eqref{eq:p456} and \eqref{eq:513}. As evident from the effective system obtained, these effective quantities depend only on the instantaneous position of the particles, their geometry, and the magnetic and flow properties of the original suspension decoded in \eqref{eq:p410}.  It is worth mentioning that this paper is not concerned with \emph{modeling issues} for the colloids with particles whose magnetic properties are described by the linear relation between the magnetic flux density $\emB$ and the magnetic field strength $\emH$ 
suspended in a viscous fluid, in the presence of a external magnetic field, which is an interesting and important topic in itself (see relevant references cited in the Introduction).  In the future, however, the authors intend to continue \emph{analyzing} the \emph{effective behavior} of the suspensions described by more complicated systems, including the nonlinear magnetic relation, the two-way coupling between the flow and the magnetic descriptions of the suspension and perhaps, the interaction between the particles and the Navier-Stokes description of the carrier fluid, whose results will be reported elsewhere.